\documentclass{article}
\usepackage{amsmath,amssymb,amsthm,graphicx,tikz,caption}

\usetikzlibrary{cd}

\tikzcdset{
  box edge/.style={dashed,
  "{}"{marking, circle, draw, solid, fill=white,
       inner sep=0pt, minimum size=4.4pt, pos=#1}},
  box edge/.default=0.5,
  boxtimes edge/.style={dashed,
                        "{}"{marking, circle, fill,
                             inner sep=0pt, minimum size=3.8pt, pos=#1}},
  boxtimes edge/.default=0.5,
}
\tikzcdset{
  inner loop/.style={loop, out=#1+30, in=#1-30, distance=5mm},
  outer loop/.style={loop, out=#1+30, in=#1-30, distance=11mm},
}

\newcommand{\ocircto}{%
  \mathrel{%
    \tikz[baseline=-0.5ex]{
      \draw[->, dashed, line width=.4pt] (0,0) -- (1.8em,0);
      \draw[fill=white] (.9em,0) circle (.8ex);
    }%
  }%
}

\newcommand{\bulletto}{%
  \mathrel{%
    \tikz[baseline=-0.5ex]{
      \draw[->, dashed, line width=.4pt] (0,0) -- (1.8em,0);
      \fill (.9em,0) circle (.8ex);
    }%
  }%
}

\title{Fixed points in abstract provability structures}
\author{Tsubasa Kumasaka\footnote{Email: 258x025x@stu.kobe-u.ac.jp}
\footnote{Graduate School of System Informatics, Kobe University, 1-1 Rokkodai, Nada, Kobe 657-8501, Japan.}
and Taishi Kurahashi\footnote{Email: kurahashi@people.kobe-u.ac.jp}
\footnote{Graduate School of System Informatics, Kobe University, 1-1 Rokkodai, Nada, Kobe 657-8501, Japan.}}
\date{}

\theoremstyle{plain}
\newtheorem{thm}{Theorem}[section]
\newtheorem*{thm*}{Theorem}
\newtheorem{lem}[thm]{Lemma}
\newtheorem{prop}[thm]{Proposition}
\newtheorem{cor}[thm]{Corollary}

\newtheorem*{fact*}{Fact}
\newtheorem{prob}[thm]{Problem}
\newtheorem*{prob*}{Problem}
\newtheorem*{cl*}{Claim}

\newtheorem*{scl*}{Subclaim}

\theoremstyle{definition}
\newtheorem{defn}[thm]{Definition}

\newcommand{\GL}{\mathbf{GL}}
\newcommand{\PR}{\mathrm{Pr}}

\begin{document}

\maketitle

\begin{abstract}
We study fixed points in abstract provability structures (APSs), which were introduced by Beklemishev and Shamkanov as an order-theoretic framework for studying G\"odel's second incompleteness theorem (G2).
For one-variable APS terms built from the provability and refutability operations $\Box$ and $\boxtimes$, we investigate the existence and uniqueness of fixed points in terms of their degree. 
The degree of a term is the number of occurrences of $\boxtimes$.
We prove that every term of degree two has a fixed point and establish a strict hierarchy among the fixed-point properties of the terms $\boxtimes\Box^k v$.
We further show that suitable levels of this hierarchy guarantee the existence and uniqueness of fixed points for arbitrary one-variable APS terms of positive degree.
We also obtain several G2-like non-refutability results, some of which depend only on whether the degree is even or odd.
Finally, we consider APSs based on meet-semilattices.
Under an additional condition, we obtain stronger fixed-point results and show that a parametrized fixed-point property implies an abstract form of L\"ob's theorem.
\end{abstract}

\section{Introduction}

The Fixed Point Theorem of formal arithmetic plays an important role in the proofs of G\"odel's incompleteness theorems (see, e.g., Boolos \cite{Bool93} and Lindstr\"om \cite{Lind03}). 
Let $T$ be a theory of arithmetic satisfying suitable conditions. 
The Fixed Point Theorem allows us to construct a sentence $\varphi$ such that $T \vdash \varphi \leftrightarrow \PR_T(\ulcorner \neg \varphi \urcorner)$, where $\PR_T(x)$ is a provability predicate of $T$.
Such a sentence $\varphi$ is used in the proof of G\"odel's first incompleteness theorem\footnote{Following Beklemishev and Shamkanov \cite{beklemishev2016some}, we call such a sentence G\"odel sentence, although it is sometimes called a Jeroslow sentence. }. 
By formalizing the argument showing the unprovability of $\neg \varphi$ in $T$, one obtains G\"odel's second incompleteness theorem (G2), which states that the consistency statement $\neg \PR_T(\ulcorner 0 = 1 \urcorner)$ is not provable in $T$. 
Moreover, one can prove L\"ob's theorem, which is a generalization of G2, stating that if $T \vdash \PR_T(\ulcorner \psi \urcorner) \to \psi$, then $T \vdash \psi$.

Interestingly, a fixed-point result can also be obtained from L\"ob's principle in the setting of provability logic.
The normal modal logic $\mathbf{GL}$ is obtained by adding the L\"ob axiom $\Box(\Box A\to A)\to\Box A$ to $\mathbf{K}$ and is the modal logic of provability (see Boolos \cite{Bool93} and Artemov and Beklemishev \cite{AB05}).
One of the early important results in provability logic is the de Jongh--Sambin Fixed Point Theorem \cite{Samb76}.
It states that, for every modal formula $A(p)$ in which every occurrence of $p$ is under the scope of $\Box$, there is a modal formula $B$ not containing $p$ such that $\GL\vdash B\leftrightarrow A(B)$.
Bernardi \cite{Bern76} also proved that such a fixed point is unique up to provable equivalence in $\GL$.
Thus, fixed-point phenomena are closely connected with L\"ob's theorem and incompleteness phenomena.

Beklemishev and Shamkanov \cite{beklemishev2016some} introduced the notion of \emph{Abstract Provability Structure} (APS) in order to clarify some of the essential structure behind G2. 
An APS is a preorder equipped with two unary operations $\Box$ and $\boxtimes$, which represent arithmetized provability and refutability, respectively. 
In ordinary arithmetic and provability logic, G2 is studied within a richer logical framework involving negation and implication. 
APSs remove most of this framework and keep only some basic interactions between provability and refutability that are needed for G2. 
Beklemishev and Shamkanov showed that even in this very weak setting, an analogue of G2 can be obtained if the existence of a G\"odelian fixed point is assumed. 
Here, a G\"odelian fixed point is an element $a$ of an APS satisfying $a = \boxtimes a$. 
In this way, APSs provide a framework for studying which parts of G2 follow only from the basic relations between provability and refutability, and conversely, which parts depend on the usual logical structure.
Their approach also clarifies the role of contraction in the proof of G2. 
They showed that, in a more general setting, the existence of G\"odelian fixed points alone does not imply G2 in the absence of contraction.

A natural question is what can be said about fixed points and G2-like phenomena using only these interactions between $\Box$ and $\boxtimes$.
In this paper, we study not only G\"odelian fixed points, but also fixed points of general one-variable terms in the language of APS.
In particular, we ask which terms have fixed points, what kinds of G2-like phenomena follow from the existence of fixed points, and what results can be obtained by enriching APSs with additional structures or conditions.
By studying these questions, we aim to clarify what can be obtained from the basic framework of APSs concerning fixed points and incompleteness.

We summarize the organization and main results of the paper.

Section 3 introduces basic notions and tools for studying fixed points in APSs.

In Section 4, we study the existence of fixed points.
We prove that every one-variable APS term of degree two has a fixed point in every APS and give explicit fixed points for such terms (Theorem \ref{2boxtimesfp}).
Here, the degree of a term is the number of occurrences of $\boxtimes$.
For terms of degree four, we obtain fixed points in several cases, while the general problem remains open.

Section 5 studies hierarchies of fixed-point properties.
For a one-variable APS term $\triangle v$, let $\mathsf{E}_S(\triangle)$ denote the existence of a fixed point of $\triangle v$ in an APS $S$. 
We show that the properties $\mathsf{E}_S(\boxtimes\Box^k)$ form the following hierarchy (Proposition \ref{prop_hierarchy}):
\[
    \mathsf{E}_S(\boxtimes)
    \Rightarrow
    \mathsf{E}_S(\boxtimes\Box)
    \Rightarrow
    \mathsf{E}_S(\boxtimes\Box^2)
    \Rightarrow\cdots,
\]
Moreover, this hierarchy is strict (Proposition \ref{prop_strict}).
We also prove that the existence of a fixed point at an appropriate level of this hierarchy guarantees both the existence and uniqueness of a fixed point for a much wider class of one-variable APS terms (Theorem \ref{thm_collapse}).
As a consequence, the existence of a G\"odelian fixed point implies that every one-variable APS term of positive degree has a unique fixed point (Corollary \ref{cor_hierarchy_godel}).

In Section 6, we investigate G2-like phenomena associated with these fixed-point properties.
We first show that $\mathsf{E}_S(\boxtimes\Box^k)$ implies the non-refutability of $\Box^k \boxtimes \top$ for every consistent APS $S$ (Proposition \ref{prop_g2_hierarchy}). 
More generally, for every one-variable APS term $\triangle v$ of positive degree, the non-refutability of $\triangle\top$ follows from $\mathsf{E}_S(\boxtimes\Box^k)$ for an appropriate $k$ (Proposition \ref{prop_ge_general}).
We then prove two general results that depend only on the parity of the degree.
For every consistent APS, if $\triangle v$ has even degree, then $\triangle\top$ is not refutable without any fixed-point assumption (Theorem \ref{thm_even_nonrefutable}).
If $\triangle v$ has odd degree, then every fixed point of $\triangle v$ is non-refutable (Theorem \ref{thm_odd_nonrefutable}).
On the other hand, we give an example showing that the existence of fixed points for all terms of degree $3$ does not in general imply the non-refutability of $\triangle\top$ for terms $\triangle v$ of odd degree (Proposition \ref{prop_no_g2}).

In Section 7, we consider APSs based on meet-semilattices.
We introduce a strengthening (C3*) of the condition (C3) of APS, which may be viewed as an abstract counterpart of the modal axiom $\mathbf{K}$.
For one-variable APS terms in the original language, we obtain a complete classification in terms of parity in meet APSs satisfying (C3*). 
That is, every term of even degree has a fixed point, while every term of odd degree has a fixed point exactly when a G\"odelian fixed point exists (Theorem \ref{thm_meet_parity}).
We then allow the meet operation in the formation of terms. 
We show that the existence of a G\"odelian fixed point already guarantees the existence of fixed points for all such one-variable meet APS terms (Proposition \ref{prop_meet_godel}).
We also obtain the existence of a fixed point for a class of meet APS terms of even degree satisfying a natural condition (Theorem \ref{thm_meet_even}).
Finally, we study an abstract version of L\"ob's theorem.
We show that a parametrized fixed-point condition for the two-variable term $\boxtimes(v\land w)$, together with (C3*), implies L\"ob's theorem (Theorem \ref{thm_lob}).
However, the converse fixed-point phenomenon fails. 
We show that there is a meet APS satisfying (C3*) and L\"ob's theorem in which no one-variable APS term of odd degree has a fixed point (Proposition \ref{prop_lob_nofp}).
This contrasts with the de Jongh--Sambin Fixed Point Theorem for $\mathbf{GL}$.

\section{Preliminaries on APSs}

In this section, we introduce \emph{abstract provability structures} and review some results of Beklemishev and Shamkanov \cite{beklemishev2016some}. 
They defined abstract provability structures using preorders. 
Throughout the present paper, since we are concerned only with their order-theoretic and fixed-point properties, we identify any two elements $x$ and $y$ such that $x\leq y$ and $y\leq x$. 
Thus, unlike in the original definition, we assume that the underlying relation is a partial order.

\begin{defn}[Abstract provability structure (APS)]
    A structure $S = \langle L, \leq, \top, \bot, \Box, \boxtimes\rangle$ is called an \emph{abstract provability structure} (APS) if the following conditions hold: 
    \begin{enumerate}
        \item $L$ is a nonempty set.
        \item $\leq$ is a partial order on $L$.
        \item $\top$ and $\bot$ are elements of $L$.
        \item $\Box$ and $\boxtimes$ are unary operations on $L$ satisfying the  following conditions for all $x, y \in L$:
        \begin{description}
            \item[(C1)] If $x \leq y$, then $\Box x \leq \Box y$ and $\boxtimes y \leq \boxtimes x$;
            \item[(C2)] $\top \leq \boxtimes \bot$;
            \item[(C3)] If $x \leq \Box y$ and $x \leq \boxtimes y$, then $x \leq \boxtimes \top$;
            \item[(C4)] $\boxtimes x \leq \Box \boxtimes x$.
        \end{description}
    \end{enumerate}
\end{defn}

The intended interpretation is that $x\leq y$ means that $y$ is derivable from $x$ over some formal system of arithmetic. 
$\Box x$ and $\boxtimes x$ are intended to represent the arithmetized provability and refutability of $x$, respectively.
The elements $\top$ and $\bot$ represent a provable sentence and a refutable sentence, respectively.
Despite the notation, we do not assume that $\top$ is the greatest element or that $\bot$ is the least element of $L$.

Condition (C1) expresses the monotonicity of provability and the antitonicity of refutability.
Condition (C2) corresponds to necessitation for provable sentences. 
Condition (C3) expresses that the combination of the provability and refutability of a sentence leads to inconsistency.
Condition (C4) says that refutability can be formally verified.

\begin{defn} 
Let $S = \langle L, \leq, \top, \bot, \Box, \boxtimes\rangle$ be an APS and let $x \in L$. 
    \begin{enumerate}  
        \item $x$ is \emph{provable} in $S$ if $\top \leq x$.
        \item $x$ is \emph{refutable} in $S$ if $x \leq \bot$.
        \item $S$ is \emph{inconsistent} if $\top \leq \bot$. 
        Otherwise, $S$ is \emph{consistent}.
         \item $x$ is a \emph{G\"odelian fixed point} of $S$ if $x = \boxtimes x$. 
   \end{enumerate}
\end{defn}

An abstract version of G2 and an abstract counterpart of its formalization are stated as follows.

\begin{thm}[{\cite[Theorem 2.5]{beklemishev2016some}}]\label{BS1}
    Suppose that an APS $S$ has a G\"odelian fixed point.
    \begin{enumerate}
        \item  If $S$ is consistent, then $\boxtimes \top$ is not refutable in $S$. 
        \item  $\boxtimes \boxtimes \top\leq \boxtimes \top$. 
    \end{enumerate}
\end{thm}

Beklemishev and Shamkanov considered the following additional natural condition: 

\begin{description}
    \item[(C5)] $x \leq \top$ for all $x \in L$. 
\end{description}

\begin{thm}[{\cite[Theorem 2.6]{beklemishev2016some}}]\label{BS2}
    Suppose that an APS $S$ satisfies \textup{(C5)}. 
    \begin{enumerate}
    \item Every G\"odelian fixed point $p$ of $S$ satisfies $p = \boxtimes \top$. 
    
    \item If a G\"odelian fixed point exists, then $\boxtimes \boxtimes \top= \boxtimes \top$.
    \end{enumerate}
\end{thm}

Beklemishev and Shamkanov subsequently studied consequence relations with implication, derivability conditions, and structural rules. 
We do not pursue this direction here.

\section{Basic properties of fixed points in APS}

In this section, we begin our analysis of fixed points in APSs.
We first give some examples showing the role of (C5) and the independence of the two operations $\Box$ and $\boxtimes$.
We then introduce terminology and notation concerning fixed points and establish some basic properties that will be used throughout the paper.

\subsection{Examples}

We first show two examples of APSs concerning the role of (C5) in Theorem \ref{BS2} of Beklemishev and Shamkanov.

\begin{prop}\label{prop_eg1}
    There is an APS which does not satisfy \textup{(C5)} and has two distinct G\"odelian fixed points. 
    So, \textup{(C5)} cannot be omitted from the uniqueness theorem of G\"odelian fixed points in general.
\end{prop}

\begin{proof} 
An example of an APS is given by a four-point model $\top \lneq p, q \lneq \bot$, where $p$ and $q$ are incomparable, $\boxtimes$ swaps $\top$ and $\bot$ and fixes $p$ and $q$, and $\Box$ is the identity. 
In this APS, $p$ and $q$ are distinct G\"odelian fixed points and are different from $\boxtimes \top$.

This APS is visualized in Figure \ref{fig1}. 
Throughout the present paper, APSs are illustrated by Hasse diagrams with arrows representing the operations $\Box$ and $\boxtimes$. 
The solid lines form the Hasse diagram of the underlying order, which increases upwards; the arrows $\ocircto$ and $\bulletto$ represent the operations $\Box$ and $\boxtimes$, respectively.

\begin{center}
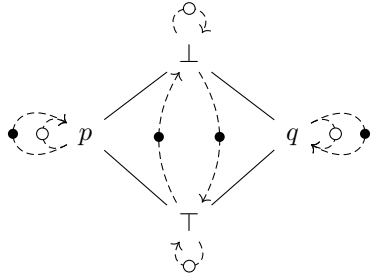

        \begin{tikzcd}
            & \bot  \arrow[rd, no head] \arrow[ld, no head] \arrow[box edge, inner loop= 90 ] \arrow[dd, boxtimes edge, bend left = 30] & \\ 
            p  \arrow[rd, no head] \arrow[box edge, inner loop = 180] \arrow[boxtimes edge, outer loop= 180 ]&  & q \arrow[ld, no head] \arrow[box edge, inner loop = 0] \arrow[boxtimes edge, outer loop= 0 ]\\
            & \top \arrow[box edge, inner loop = 270]  \arrow[uu, boxtimes edge, bend left = 30]& 
        \end{tikzcd}
\captionof{figure}{APS for Proposition \ref{prop_eg1}}
\label{fig1}
\end{center}
\end{proof}

The following proposition shows that the assumption of the existence of a G\"odelian fixed point in the statement of Theorem \ref{BS2}(2) cannot be omitted.

\begin{prop}\label{prop_eg2}
    There is an APS which satisfies \textup{(C5)} but does not have a G\"odelian fixed point.
\end{prop}

\begin{proof}
An example of an APS is given by the two-point model $\bot \lneq \top$, where $\boxtimes$ swaps $\top$ and $\bot$, and $\Box$ is the identity. 
We have $\boxtimes \boxtimes \top \neq \boxtimes \top$. 

\begin{center}
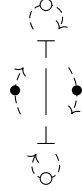

        \begin{tikzcd}
             \top \arrow [d, no head] \arrow[box edge, inner loop = 90] \arrow[boxtimes edge, d , bend left=30, shift left = 2]\\ 
             \bot \arrow[box edge, inner loop = 270] \arrow[boxtimes edge, u, bend left = 30, shift left = 2 ]\\
        \end{tikzcd}
\captionof{figure}{APS for Proposition \ref{prop_eg2}}
\label{fig2}
\end{center}
\end{proof}

The next example shows that the two operations $\Box$ and $\boxtimes$ of APS should in general be regarded as independent, even when the underlying partial order is a Boolean algebra.

\begin{prop}\label{prop_eg3}
   There is an APS satisfying \textup{(C5)} based on a Boolean algebra such that there is an element $x$ with $\boxtimes x \neq \Box \neg x$. 
\end{prop}

\begin{proof} The desired APS is given in Figure~\ref{fig3}. We have $\boxtimes p = \bot \neq \neg p = \Box \neg p$.

\begin{center}
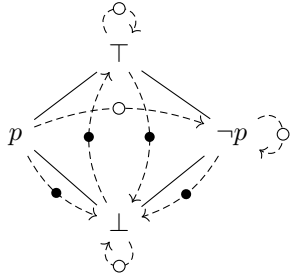

        \begin{tikzcd}
            & \top  \arrow[rd, no head] \arrow[ld, no head] \arrow[box edge, inner loop= 90 ] \arrow[dd, boxtimes edge, bend left = 30]& \\ 
            p  \arrow[rd, no head] \arrow[box edge, rr, bend left = 20] \arrow[boxtimes edge, rd, bend right = 20 ]&  & \neg p \arrow[ld, no head] \arrow[box edge, inner loop = 0] \arrow[boxtimes edge, bend left = 20 , ld]\\
            & \bot \arrow[box edge, inner loop = 270]  \arrow[uu, boxtimes edge, bend left = 30]& 
        \end{tikzcd}
\captionof{figure}{APS for Proposition \ref{prop_eg3}}\label{fig3}
\end{center}

    We note that this example also shows that APS does not contain the K axiom for $\Box$, since $\Box p \land \Box \neg p \not\leq \Box (p \land \neg p) $. 
\end{proof}

\subsection{Some basic properties}

From now on, unless explicitly stated otherwise, we assume that all APSs satisfy (C5). 
In this subsection, we collect several basic properties of APS and introduce some basic notions used throughout the present paper.

The following proposition is one of the basic properties of APS, which will be used repeatedly.

\begin{prop}\label{propbasicproperty}
Let $S = \langle L, \leq, \top, \bot, \Box, \boxtimes\rangle$ be an APS.
Then, the following statements hold: 
    \begin{enumerate}
        \item $\top = \boxtimes \bot = \Box \top$.

        \item For every $x \in L$, we have $\boxtimes \top \leq \boxtimes x \leq \Box^k \boxtimes x \leq \Box^{k+1} \boxtimes x$ for all $k \geq 0$. 
    \end{enumerate}
\end{prop}
\begin{proof}
1. $\top = \boxtimes \bot$ follows from (C2) and (C5). 
Then we have $\Box \top = \Box \boxtimes \bot$. 
Thus we obtain $\top = \boxtimes \bot \leq \Box \boxtimes \bot = \Box \top \leq \top$, where the first and second inequalities follow from (C4) and (C5), respectively. 

\medskip

2. By (C5), $x \leq \top$, and hence, by (C1),
$\boxtimes \top \leq \boxtimes x$.
By (C4), $\boxtimes x \leq \Box \boxtimes x$.
Repeated applications of (C1) then yield
$\Box^k \boxtimes x \leq \Box^{k+1} \boxtimes x$ for every $k \geq 0$.
\end{proof}

A \emph{one-variable APS term} is generated from the variable $v$ by finitely many applications of $\Box$ and $\boxtimes$. 
We regard each such one-variable APS term $\triangle v$ as a unary operation on $L$.
If $a \in L$, we then write $\triangle a$ for the value of the operation $\triangle$ at $a$.
We write compositions by concatenation, so that $\triangle_0 \triangle_1 v$ means $\triangle_0 (\triangle_1 v)$. 
For each natural number $n$, $\Box^n$ and $\boxtimes^n$ denote the $n$-fold iterations of $\Box$ and $\boxtimes$, respectively. 

\begin{defn}
The \emph{degree} $d(\triangle)$ of a one-variable APS term $\triangle v$ is the number of occurrences of $\boxtimes$ in $\triangle$. 
\end{defn}

Throughout the present paper, the parity of the degree plays an important role. 
Since (C1) states that $\Box$ and $\boxtimes$ are order-preserving and order-reversing, respectively, we obtain the following proposition which is easily proved by induction on the length of terms.

\begin{prop}\label{prop_parity}
Let $S$ be an APS and let $\triangle v$ be a one-variable APS term. 
\begin{enumerate}
    \item If $d(\triangle)$ is even, then $\triangle v$ is an order-preserving operation on $L$. 
    That is, for any $x, y \in L$, $x \leq y$ implies $\triangle x \leq \triangle y$. 

    \item If $d(\triangle)$ is odd, then $\triangle v$ is an order-reversing operation on $L$. 
    That is, for any $x, y \in L$, $x \leq y$ implies $\triangle y \leq \triangle x$. 
\end{enumerate}
\end{prop}

\begin{defn}
Let $S = \langle L, \leq, \top, \bot, \Box, \boxtimes\rangle$ be an APS and $\triangle v$ be a one-variable APS term.
\begin{enumerate}
    \item An element $p \in L$ is called a \emph{fixed point} of $\triangle v$ in $S$ if $\triangle p = p$. 

    \item The statement ``$\triangle v$ has a fixed point in $S$'' is denoted by $\mathsf{E}_S(\triangle)$.
    Here $\mathsf{E}$ stands for `existence'. 

    \item The statement ``$\triangle v$ has at most one fixed point in $S$'' is denoted by $\mathsf{U}_S(\triangle)$.
    Here $\mathsf{U}$ stands for `uniqueness'. 
\end{enumerate}
\end{defn}

The following method of ``circulation'' will be used repeatedly.
It allows us to move an initial part of a term to the end without changing the existence or uniqueness of its fixed points.

\begin{prop}\label{circulation}
Let $S$ be an APS and let $\triangle_0 v$ and $\triangle_1 v$ be one-variable APS terms. 
\begin{enumerate}
    \item $\mathsf{E}_S(\triangle_0 \triangle_1) \iff \mathsf{E}_S(\triangle_1 \triangle_0)$. 

    \item $\mathsf{U}_S(\triangle_0 \triangle_1) \iff \mathsf{U}_S(\triangle_1 \triangle_0)$. 
\end{enumerate}
\end{prop}
\begin{proof}
1. Suppose that $\mathsf{E}_S(\triangle_0\triangle_1)$ holds and let $p$ be a fixed point of $\triangle_0\triangle_1 v$. 
Then, $\triangle_1\triangle_0(\triangle_1 p) = \triangle_1(\triangle_0\triangle_1 p) = \triangle_1 p$. 
Thus, $\triangle_1 p$ is a fixed point of $\triangle_1\triangle_0 v$, and so
$\mathsf{E}_S(\triangle_1\triangle_0)$ holds.
The converse follows by exchanging $\triangle_0$ and $\triangle_1$.

\medskip

2. Suppose that $\mathsf{U}_S(\triangle_0\triangle_1)$ holds.
Let $p$ and $q$ be fixed points of $\triangle_1\triangle_0 v$.
As in (1), we obtain $\triangle_0\triangle_1(\triangle_0 p) = \triangle_0 p$ and $\triangle_0\triangle_1(\triangle_0 q) = \triangle_0 q$.
By the uniqueness assumption, we have $\triangle_0 p=\triangle_0 q$.
Therefore, $p = \triangle_1\triangle_0 p = \triangle_1\triangle_0 q = q$.
Hence $\mathsf{U}_S(\triangle_1\triangle_0)$ holds.
The converse follows by exchanging $\triangle_0$ and $\triangle_1$.
\end{proof}

We next establish some elementary relations between the fixed-point properties of a term $\triangle v$ and those of its self-composition $\triangle\triangle v$.

\begin{prop}\label{prop_double}
Let $S$ be an APS and let $\triangle v$ be a one-variable APS term. 
Then 
\begin{enumerate}
    \item $\mathsf{E}_S(\triangle) \Rightarrow \mathsf{E}_S(\triangle \triangle)$. 
    
    \item $\mathsf{U}_S(\triangle \triangle) \ \&\ \mathsf{E}_S(\triangle \triangle) \Rightarrow \mathsf{U}_S(\triangle) \ \&\ \mathsf{E}_S(\triangle)$.

    \item $d(\triangle)$ is odd $\&\ \triangle \triangle \boxtimes \top = \triangle \boxtimes \top \Rightarrow \mathsf{U}_S(\triangle \triangle)$.
\end{enumerate}
\end{prop}
\begin{proof}
1. Suppose that $\mathsf{E}_S(\triangle)$ holds. 
Then, we find a fixed point $p$ of $\triangle v$. 
That is, $\triangle p = p$. 
Then $\triangle \triangle p = \triangle p = p$. 
Thus, $\mathsf{E}_S(\triangle \triangle)$ holds. 

\medskip

2. Suppose that $\mathsf{U}_S(\triangle \triangle)$ and $\mathsf{E}_S(\triangle \triangle)$ hold. 
Let $p$ be a fixed point of $\triangle \triangle v$. 
Then $\triangle \triangle p = p$. 
We have $\triangle \triangle \triangle p = \triangle p$, and hence $\triangle p$ is also a fixed point of $\triangle \triangle v$. 
From the uniqueness, we obtain $p = \triangle p$. 
Thus $\mathsf{E}_S(\triangle)$ holds. 
Since the proof of (1) shows that every fixed point of $\triangle v$ is also a fixed point of $\triangle \triangle v$, we also obtain that $\mathsf{U}_S(\triangle)$ holds. 

\medskip

3. Suppose that the degree of $\triangle v$ is odd and $\triangle \triangle \boxtimes \top = \triangle \boxtimes \top$ holds. 
Let $p \in L$ be any fixed point of $\triangle \triangle v$, that is, $\triangle \triangle p = p$. 
Since $\triangle v$ is of the form $\Box^k \boxtimes \triangle_0 v$, it follows from Proposition \ref{propbasicproperty}(2) that $\boxtimes \top \leq \triangle p$ and $\boxtimes \top \leq \triangle \triangle p = p$. 
Since $\triangle v$ is an order-reversing operation by Proposition \ref{prop_parity}(2), we have $p = \triangle \triangle p \leq \triangle \boxtimes \top$ and $\triangle p \leq \triangle \boxtimes \top$. 
The latter inequality yields $\triangle \boxtimes \top = \triangle \triangle \boxtimes \top \leq \triangle \triangle p = p$. 
Hence, we obtain $p = \triangle \boxtimes \top$. 
Since every fixed point of $\triangle \triangle v$ is equal to $\triangle \boxtimes \top$, we conclude that $\mathsf{U}_S(\triangle \triangle)$ holds. 
\end{proof}

\section{Existence of fixed points}\label{Sec_existence}

In this section, we study which one-variable APS terms have fixed points in every APS.
As shown in Proposition \ref{prop_parity}, the parity of the degree provides a natural distinction between one-variable APS terms. 
For terms of odd degree, fixed points are not guaranteed in general.
Indeed, as we show below, there is an APS in which no term of odd degree has a fixed point.
On the other hand, every term of even degree induces an order-preserving operation on APSs. 
Thus, if an APS $S$ is based on a complete lattice, then the Knaster--Tarski Fixed-Point Theorem (cf.~\cite{DP02}) guarantees that every term of even degree has a fixed point in $S$. 
In particular, if an APS $S$ is based on a finite lattice, then every term of even degree has a fixed point in $S$. 
This naturally raises the question whether the same conclusion holds for arbitrary APSs. 

Terms of degree zero always have fixed points, since every such term is of the form $\Box^k v$ and $\Box\top=\top$ by Proposition \ref{propbasicproperty}(1).
We first show that every term of degree two also has a fixed point.
We then show that fixed points are guaranteed for some terms of degree four. 
For arbitrary terms of even degree at least four, however, we do not know whether fixed points are guaranteed in general. 

\subsection{Terms of degree two}

Every one-variable APS term of degree two is of the form
\[
    \Box^k\boxtimes\Box^m\boxtimes\Box^n v
\]
for some natural numbers $k,m,n$. 
The following theorem shows that every such term has a fixed point. 
Moreover, a fixed point can be explicitly given in terms of $k,m,n$.

\begin{thm}\label{2boxtimesfp}
Let $S$ be an APS and let $\triangle v$ be a one-variable APS term of degree two. 
Then $\mathsf{E}_S(\triangle)$ holds. 
More precisely, the following statements hold.  
Let $k, m, n$ be natural numbers. 
 \begin{enumerate}
        \item If $m \geq n$, then $\boxtimes \Box^m \boxtimes \Box^n v$ has a fixed point $\boxtimes \top$.
        
        \item If $m < n$, then $\boxtimes \Box^m \boxtimes \Box^n v$ has a fixed point $\boxtimes \Box^m \boxtimes \top$.
        
        \item If $m \geq n + k$, then $\Box^k \boxtimes \Box^m \boxtimes \Box^n v$ has a fixed point $\Box^k \boxtimes \top$.
        
        \item If $m < n + k$, then $\Box^k \boxtimes \Box^m \boxtimes \Box^n v$ has a fixed point $ \Box^k \boxtimes\Box^m \boxtimes \top$. 
    \end{enumerate}
\end{thm}
\begin{proof}
1. Suppose that $m \geq n$. 
Since Proposition \ref{propbasicproperty}(2) gives $\boxtimes \top \leq\boxtimes \Box^m \boxtimes \Box^n \boxtimes \top$, it suffices to show $\boxtimes \Box^m \boxtimes \Box^n \boxtimes \top\leq \boxtimes \top $. 

Using Proposition \ref{propbasicproperty}(2) again, we have $\boxtimes \top \leq \Box^{m-n} \boxtimes \Box^n \boxtimes \top$. 
Applying (C1) $n$ times, $\Box^n \boxtimes \top \leq\Box^{m} \boxtimes \Box^n \boxtimes \top$, and so $\boxtimes \Box^{m} \boxtimes \Box^n \boxtimes \top\leq \boxtimes \Box^n \boxtimes \top$ by (C1) again. 
Since $\boxtimes \Box^n \boxtimes \top \leq \Box \Box^{m}\boxtimes \Box^n \boxtimes \top$ by Proposition \ref{propbasicproperty}(2), we have 
\[
    \boxtimes \Box^{m} \boxtimes \Box^n \boxtimes \top\leq \Box \Box^{m}\boxtimes \Box^n \boxtimes \top. 
\]
Since $\boxtimes \Box^{m} \boxtimes \Box^n \boxtimes \top \leq \boxtimes \Box^{m} \boxtimes \Box^n \boxtimes \top$, we conclude $\boxtimes \Box^{m} \boxtimes \Box^n \boxtimes\top \leq \boxtimes \top$ by (C3). 

\medskip

2. Suppose that $m<n$. 
By (1), $\boxtimes \Box^n \boxtimes \Box^m \boxtimes \top= \boxtimes \top$. 
Applying the operation $\boxtimes \Box^m v$ to both sides, we obtain $\boxtimes \Box^m\boxtimes \Box^n (\boxtimes \Box^m \boxtimes \top) = \boxtimes \Box^m\boxtimes \top$. 

\medskip

3. Suppose that $m \geq n+k$. 
We have $\boxtimes \Box^m \boxtimes \Box^{n+k} \boxtimes \top = \boxtimes \top$ from (1). 
We then have $\Box^k\boxtimes \Box^m \boxtimes \Box^n (\Box^k \boxtimes \top) = \Box^k\boxtimes \top$. 

\medskip

4. Suppose that $m < n+k$. 
From (2), we have $\boxtimes \Box^m\boxtimes \Box^{n+k} \boxtimes \Box^m \boxtimes \top= \boxtimes \Box^m\boxtimes \top $. 
We then obtain $\Box^k\boxtimes \Box^m\boxtimes \Box^n (\Box^k \boxtimes \Box^m \boxtimes \top) = \Box^k \boxtimes \Box^m\boxtimes \top$. 
\end{proof}

Theorem \ref{2boxtimesfp} guarantees the existence of fixed points for terms of degree two, but it does not guarantee their uniqueness.

\begin{prop}
There is an APS $S$ satisfying the following conditions:
\begin{enumerate}
    \item Every one-variable APS term of even degree has two distinct fixed points in $S$. 

    \item Every one-variable APS term of odd degree has no fixed point in $S$. 
\end{enumerate}
\end{prop}
\begin{proof}
Consider the two-element APS from Proposition \ref{prop_eg2}, illustrated in Figure \ref{fig2}.
Since $\Box$ is the identity and $\boxtimes$ interchanges $\bot$ and $\top$, every one-variable APS term of even degree is the identity map on $L$, and every one-variable APS term of odd degree interchanges $\bot$ and $\top$.
Hence every term of even degree has the two distinct fixed points $\bot$ and $\top$. 
Also no term of odd degree has a fixed point.
\end{proof}

\subsection{Terms of degree four}

The proof of Theorem \ref{2boxtimesfp} gives a complete analysis of the existence of fixed points for terms of degree two. 
It is natural to ask whether a similar argument works for terms of higher even degree. 
Already for degree four, however, a difficulty appears.

Every one-variable APS term of degree four is of the form
\[ 
    \Box^{k_0} \boxtimes\Box^{k_1} \boxtimes\Box^{k_2} \boxtimes\Box^{k_3} \boxtimes\Box^{k_4}v
\]
for some $k_0, k_1, k_2, k_3, k_4$. 
Applying Proposition \ref{circulation}, the existence of a fixed point of this term is equivalent to that of a fixed point of the following term: 
\[ 
    \boxtimes\Box^{k_1} \boxtimes\Box^{k_2} \boxtimes\Box^{k_3} \boxtimes\Box^{k_4 + k_0}v
\]
Moreover, applying Proposition \ref{circulation} again, it suffices to consider a term of the form
\[
    \boxtimes\Box^{k_1} \boxtimes\Box^{k_2} \boxtimes\Box^{k_3} \boxtimes\Box^{k_4}v
\]
where $k_1$ is maximal among $k_1, k_2, k_3, k_4$.
The possible orderings of $k_2, k_3$ and $k_4$ can be arranged in Table \ref{table_degree_four}. 

\begin{table}[ht]
\centering
\begin{tabular}{|c|c|c|}
 \hline
 & ordering & existence \\ 
 \hline
 1 & $k_4 \leq k_3 \leq k_2 \leq k_1$ & $\checkmark$ \\
 2 & $k_4 \leq k_2 \leq k_3 \leq k_1$ & $\checkmark$ \\
 3 & $k_2 \leq k_4 \leq k_3 \leq k_1$ & $\checkmark$ \\
 4 & $k_3 \leq k_4 \leq k_2 \leq k_1$ & ? \\
 5 & $k_3 \leq k_2 \leq k_4 \leq k_1$ & ? \\
 6 & $k_2 \leq k_3 \leq k_4 \leq k_1$ & ? \\ 
 \hline
\end{tabular}
\caption{Classification of terms of degree four}\label{table_degree_four}
\end{table}

In each of cases 1, 2, and 3 in Table \ref{table_degree_four}, we have that $k_1\geq k_2$ and $k_3\geq k_4$. 
Therefore it follows from Theorem \ref{2boxtimesfp}(1) that 
\[ \boxtimes\Box^{k_1}\boxtimes\Box^{k_2} \boxtimes \top = \boxtimes \top \quad \text{and} \quad \boxtimes\Box^{k_3}\boxtimes\Box^{k_4} \boxtimes \top = \boxtimes \top. 
\]
Therefore, 
\[
    \boxtimes\Box^{k_1} \boxtimes\Box^{k_2} \boxtimes\Box^{k_3} \boxtimes\Box^{k_4} \boxtimes \top = \boxtimes \top.
\]
Thus $\boxtimes \top$ is a fixed point.
On the other hand, in the remaining three cases, this argument does not apply. 

\begin{prob}\label{prob_even_degree}
For every one-variable APS term $\triangle$ of degree four, does $\mathsf{E}_S(\triangle)$ hold for all APSs $S$?
More generally, for every one-variable APS term $\triangle$ of even degree, does $\mathsf{E}_S(\triangle)$ hold for all APSs $S$?
\end{prob}

We conjecture that the first question already has a negative answer.

\section{Hierarchies of fixed-point properties}

In this section, we compare the fixed-point properties of one-variable APS terms.
We begin with terms of degree one. 
Every such term is of the form $\Box^m\boxtimes\Box^n v$ for some natural numbers $m, n$.
By Proposition \ref{circulation}, the existence and uniqueness of fixed points of this term are equivalent to those of $\boxtimes\Box^{m+n}v$.
Thus, it suffices to consider terms of the form $\boxtimes\Box^k v$.

The following proposition gives the basic hierarchy.

\begin{prop}\label{prop_hierarchy}
Let $S$ be any APS. 
For every natural number $k$, the following statements hold.
\begin{enumerate}
\item $\mathsf{E}_S(\boxtimes\Box^k) \iff \boxtimes\Box^k\boxtimes\top=\boxtimes\top$. 

\item $\mathsf{E}_S(\boxtimes\Box^k) \iff \mathsf{U}_S(\boxtimes\Box^k\boxtimes\Box^k)$.

\item $\mathsf{E}_S(\boxtimes\Box^k) \Rightarrow \mathsf{E}_S(\boxtimes\Box^{k+1})$.
\end{enumerate}
\end{prop}

\begin{proof}
1. $(\Rightarrow)$: 
Suppose that $\mathsf{E}_S(\boxtimes \Box^k)$ holds. 
Let $p \in L$ be a fixed point of $\boxtimes \Box^k v$, that is, $\boxtimes\Box^k p = p$.
By Proposition \ref{propbasicproperty}(2), $\boxtimes\top \leq\boxtimes\Box^k p=p$.

On the other hand, by Proposition \ref{propbasicproperty}(2), $p = \boxtimes\Box^k p \leq \Box^{k+1} \boxtimes \Box^k p = \Box \Box^k p$.
Since $p \leq \boxtimes\Box^k p$, (C3) gives $p\leq\boxtimes\top$.
Therefore $p=\boxtimes\top$.
Hence $\boxtimes\Box^k\boxtimes\top=\boxtimes\top$. 

\medskip

$(\Leftarrow)$: Trivial. 

\medskip

2. $(\Rightarrow)$: Suppose that $\mathsf{E}_S(\boxtimes\Box^k)$ holds.
By (1), $\boxtimes\Box^k\boxtimes \top=\boxtimes\top$.
Applying $\boxtimes \Box^k$ to both sides, we also have $\boxtimes\Box^k\boxtimes\Box^k\boxtimes\top = \boxtimes\Box^k\boxtimes\top$. 
Since the degree of $\boxtimes\Box^k$ is odd, by Proposition \ref{prop_double}(3), we conclude that $\mathsf{U}_S(\boxtimes\Box^k\boxtimes\Box^k)$ holds. 

\medskip

$(\Leftarrow)$: Suppose that $\mathsf{U}_S(\boxtimes\Box^k\boxtimes\Box^k)$ holds.
Since the degree of $\boxtimes \Box^k\boxtimes \Box^k v$ is two, by Theorem~\ref{2boxtimesfp}, $\mathsf{E}_S(\boxtimes \Box^k\boxtimes \Box^k)$ holds. 
Then by Proposition \ref{prop_double}(2), we conclude that $\mathsf{E}_S(\boxtimes \Box^k)$ holds. 

\medskip

3. Suppose that $\mathsf{E}_S(\boxtimes\Box^k)$ holds.
By (1), $\boxtimes\Box^k\boxtimes\top=\boxtimes\top$.
By (C4) and (C1), $\Box^k\boxtimes\top \leq \Box^{k+1}\boxtimes\top$.
Therefore, by (C1), $\boxtimes\Box^{k+1}\boxtimes\top \leq \boxtimes\Box^k\boxtimes\top = \boxtimes\top$.
Since $\boxtimes \top \leq \boxtimes\Box^{k+1}\boxtimes\top$ by Proposition \ref{propbasicproperty}(2), we obtain $\boxtimes\Box^{k+1}\boxtimes\top=\boxtimes\top$. 
Thus $\mathsf{E}_S(\boxtimes\Box^{k+1})$ holds.
\end{proof}

From the proof of Proposition \ref{prop_hierarchy}, we obtain the following corollary. 

\begin{cor}\label{cor_hierarchy}
Let $S$ be any APS and $k \geq 0$. 
Suppose that $\mathsf{E}_S(\boxtimes \Box^k)$ holds. 
Then $\mathsf{U}_S(\boxtimes \Box^k)$ holds. 
Furthermore, $\boxtimes\top$ is the unique fixed point of $\boxtimes \Box^k v$.
\end{cor}

Proposition \ref{prop_hierarchy} can be summarized by the following hierarchy. 

\[
\begin{array}{ccccc}
\boxtimes \boxtimes \top = \boxtimes \top & \iff & \mathsf{E}_S(\boxtimes) & \iff & \mathsf{U}_S(\boxtimes^2) \\
\Downarrow & & & & \\
\boxtimes \Box \boxtimes \top = \boxtimes \top & \iff & \mathsf{E}_S(\boxtimes \Box) & \iff & \mathsf{U}_S(\boxtimes \Box \boxtimes \Box)  \\
\Downarrow & & & & \\
\boxtimes \Box^2 \boxtimes \top = \boxtimes \top & \iff & \mathsf{E}_S(\boxtimes \Box^2) & \iff& \mathsf{U}_S(\boxtimes \Box^2 \boxtimes \Box^2) \\
\Downarrow & & & & \\ 
\vdots & & & &
\end{array}
\]

The implications in this hierarchy are all strict.

\begin{prop}\label{prop_strict}
For every natural number $k$, there exists an APS $S$ such that $\mathsf{E}_S(\boxtimes\Box^{k+1})$ holds but $\mathsf{E}_S(\boxtimes\Box^k)$ does not.
\end{prop}

\begin{proof}
Fix a natural number $k$.
Consider the structure given in Figure~\ref{fig4}.

\begin{center}
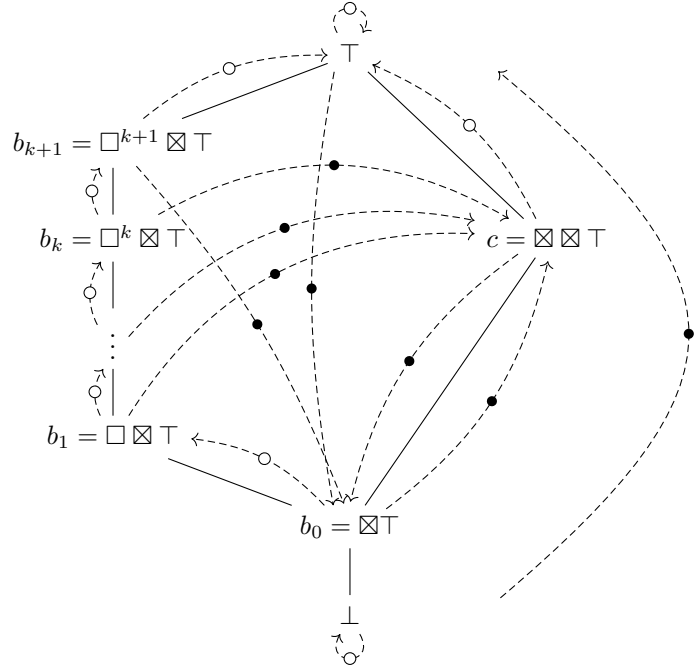


        \begin{tikzcd}
    & \top \arrow[ld, no head] \arrow[rdd, no head] \arrow[box edge, inner loop = 90]
      \arrow[ddddd, boxtimes edge, bend right = 10, shift right = 2]& \\
    b_{k+1} = \Box^{k+1} \boxtimes \top \arrow[no head, d]
      \arrow[box edge, ur, bend left = 20]
      \arrow[boxtimes edge, ddddr, bend left = 10, shift left = 0.5]&  &  \\
    b_{k} = \Box^{k}\boxtimes \top \arrow[d, no head] \arrow[box edge, u, bend left = 30]
      \arrow[boxtimes edge, rr, bend left = 30]& & c = \boxtimes \boxtimes \top
      \arrow[no head, dddl] \arrow[boxtimes edge, dddl, bend right = 20, shift right = 1]
      \arrow[box edge, luu, bend right = 20]\\
    \vdots \arrow[boxtimes edge, rru, bend left = 30] \arrow[d, no head]
      \arrow[box edge, u, bend left = 30]  & & \\
    b_1 = \Box \boxtimes \top \arrow[no head, rd] \arrow[box edge, u, bend left = 30]
      \arrow[boxtimes edge, rruu, bend left = 30]&  &   \\
     & b_0 = \boxtimes \top \arrow[no head, d] \arrow[box edge, ul, bend right = 20]
       \arrow[boxtimes edge, uuur, bend right = 20, shift right = 1]& \\
     & \bot \arrow[boxtimes edge, bend right = 50, uuuuuu, shift right = 20, looseness = 1.6] \arrow[box edge, inner loop = 270]&
\end{tikzcd}

\captionof{figure}{APS for Proposition \ref{prop_strict}}\label{fig4}
\end{center}

Here, $\top$ is the greatest element, $\bot$ is the least element, $b_0 < b_1 < \cdots < b_{k+1}$ forms a chain, and $c$ lies above $b_0$ but is incomparable with each of $b_1, \dots, b_{k+1}$. 

The operations are defined as follows:
\begin{itemize}
    \item $\Box \top = \top$ and $\boxtimes \top = b_0$;
    \item $\Box \bot = \bot$ and $\boxtimes \bot = \top$;
    \item $\Box b_i = b_{i+1}$ and $\boxtimes b_i = c$ for $0 \leq i \leq k$;
    \item $\Box b_{k+1} = \top$ and $\boxtimes b_{k+1} = b_0$;
    \item $\Box c = \top$ and $\boxtimes c = b_0$.
\end{itemize}

It is verified that this structure is an APS.
In this APS, $\boxtimes\Box^{k+1}\boxtimes\top = \boxtimes b_{k+1} = b_0 = \boxtimes\top$ and then $\mathsf{E}_S(\boxtimes\Box^{k+1})$ holds. On the other hand, $\boxtimes\Box^{k}\boxtimes\top = \boxtimes b_k = c \neq b_0 = \boxtimes\top$ and so $\mathsf{E}_S(\boxtimes\Box^{k})$ fails by Proposition \ref{prop_hierarchy}(1).
\end{proof}

The hierarchy above plays a central role in the analysis of fixed points of general one-variable APS terms.
Indeed, the following theorem, together with Proposition \ref{circulation}, shows that the existence and uniqueness of a fixed point of a term of positive degree can be derived from an appropriate level of this hierarchy.

\begin{thm}\label{thm_collapse}
Let $S$ be an APS and let $\triangle v$ be a one-variable APS term of the form $\boxtimes \Box^{k_1} \boxtimes \Box^{k_2} \cdots \boxtimes \Box^{k_n} v$. 
Let $k = \min \{k_1, k_2, \ldots, k_n\}$.
Suppose that $\mathsf{E}_S(\boxtimes \Box^k)$ holds. 
Then $\mathsf{E}_S(\triangle)$ and $\mathsf{U}_S(\triangle)$ hold. 
Furthermore, $\boxtimes\top$ is the unique fixed point of $\triangle v$.
\end{thm}

\begin{proof}
Suppose that $\mathsf{E}_S(\boxtimes \Box^k)$ holds. 
Since $k\leq k_i$ for every $i \in \{1, 2, \ldots, n\}$, by Proposition \ref{prop_hierarchy}(3), we have that $\mathsf{E}_S(\boxtimes\Box^{k_i})$ holds.  
By Proposition \ref{prop_hierarchy}(1), $\boxtimes \Box^{k_i} \boxtimes \top = \boxtimes \top$ for every $i\in\{1,\ldots,n\}$. 
Then, by applying these identities, we obtain
\[
    \boxtimes \Box^{k_1} \boxtimes \Box^{k_2} \cdots \boxtimes \Box^{k_n} \boxtimes \top = \boxtimes \top. 
\]
Thus $\boxtimes\top$ is a fixed point of $\triangle v$, and so $\mathsf{E}_S(\triangle)$ holds.

Next, we prove uniqueness.
Let $p$ be any fixed point of $\triangle v$, that is, $\triangle p = p$. 
By Proposition \ref{propbasicproperty}(2), 
\[
\boxtimes\top \leq \boxtimes\Box^{k_1} \boxtimes\Box^{k_2}\cdots \boxtimes\Box^{k_n}p = p. 
\]
Let
\[
    q : = \begin{cases} p & \text{if}\ n = 1, \\ \boxtimes\Box^{k_2}\cdots
\boxtimes\Box^{k_n}p & \text{if}\ n > 1. \end{cases}
\]
In either case, it follows from Proposition \ref{propbasicproperty}(2) that $\boxtimes\top\leq q$.
Since $\boxtimes\Box^{k_1} v$ is order-reversing, we obtain
\[
    p = \boxtimes\Box^{k_1}q \leq \boxtimes\Box^{k_1}\boxtimes\top = \boxtimes\top.
\]
Therefore $p=\boxtimes\top$.
Hence $\boxtimes\top$ is the unique fixed point of $\triangle v$,
and consequently $\mathsf{U}_S(\triangle)$ holds.
\end{proof}

\begin{cor}\label{cor_hierarchy_godel}
Let $S$ be an APS and let $\triangle v$ be a one-variable APS term with $d(\triangle) \geq 1$.
\[
    \mathsf{E}_S(\boxtimes) \Rightarrow \mathsf{E}_S(\triangle)\ \&\ \mathsf{U}_S(\triangle). 
\]
\end{cor}
\begin{proof}
By Proposition \ref{circulation}, it suffices to consider a term of the form
$\boxtimes\Box^{k_1}\boxtimes\Box^{k_2}\cdots\boxtimes\Box^{k_n}v$.
Let $k=\min\{k_1,\ldots,k_n\}$.
If $\mathsf{E}_S(\boxtimes)$ holds, then Proposition \ref{prop_hierarchy}(3) implies that $\mathsf{E}_S(\boxtimes\Box^k)$ holds.
Hence Theorem \ref{thm_collapse} yields that $\mathsf{E}_S(\triangle)$ and $\mathsf{U}_S(\triangle)$ hold.
\end{proof}

The implication in Theorem \ref{thm_collapse} cannot be reversed.

\begin{prop}
    For every natural number $k$, let $\triangle v := \boxtimes \Box^{k+1} \boxtimes \Box^{k} v$. 
    Then, there exists an APS $S$ such that $\mathsf{E}_S(\triangle)$ and $\mathsf{U}_S(\triangle)$ hold, but $\mathsf{E}_S(\boxtimes \Box^{k})$ fails.
\end{prop}

\begin{proof}
    Fix a natural number $k$. 
    Let $S$ be the APS in Figure~\ref{fig4} for $k$.  
    It is easy to check that $\triangle x = \boxtimes \top$ for every $x \in L$, and hence $\boxtimes \top$ is the unique fixed point of $\triangle v$. 
    In particular, $\mathsf{E}_S(\triangle)$ and $\mathsf{U}_S(\triangle)$ hold.
    On the other hand, $S$ separates $\mathsf{E}_S(\boxtimes \Box^{k+1})$ from $\mathsf{E}_S(\boxtimes \Box^k)$ by Proposition~\ref{prop_strict},
    and hence $\mathsf{E}_S(\boxtimes \Box^k)$ fails.
\end{proof}

\section{Fixed points and G2-like phenomena}

Abstract provability structures were introduced in order to capture abstract forms of G\"odel's second incompleteness theorem (G2).
In this section, we show that APSs can also be used to analyze a wider range of G2-like phenomena associated with general fixed-point properties.

We first consider the fixed-point hierarchy studied in the previous section and show that its levels yield corresponding non-refutability results.
The following proposition is a generalization of Theorem \ref{BS1}(1) of Beklemishev and Shamkanov. 

\begin{prop}\label{prop_g2_hierarchy}
Let $S$ be a consistent APS and let $k \geq 0$.
Suppose that $\mathsf{E}_S(\boxtimes\Box^k)$ holds.
Then $\Box^k\boxtimes\top$ is not refutable in $S$.
\end{prop}

\begin{proof}
Suppose, toward a contradiction, that $\Box^k\boxtimes\top\leq\bot$.
By (C1), $\boxtimes\bot\leq\boxtimes\Box^k\boxtimes\top$.
Since $\mathsf{E}_S(\boxtimes\Box^k)$ holds, Proposition \ref{prop_hierarchy}(1) gives
$\boxtimes\Box^k\boxtimes\top=\boxtimes\top$. 
Moreover, by (C2) and Proposition \ref{propbasicproperty}(2),
\[
    \top\leq\boxtimes\bot \leq\boxtimes\top \leq\Box^k\boxtimes\top \leq\bot.
\]
Thus $S$ is inconsistent, a contradiction.
\end{proof}

This proposition can be extended to arbitrary one-variable APS terms of positive degree.

\begin{prop}\label{prop_ge_general}
Let $S$ be a consistent APS and let $\triangle v$ be a one-variable APS term of the form
\[
    \triangle v = \Box^{k_0}\boxtimes\Box^{k_1}\boxtimes\cdots \boxtimes\Box^{k_n}v,
\]
where $n\geq 1$.
Let $k=\min\{k_0,k_1,\ldots,k_{n-1}\}$.
Suppose that $\mathsf{E}_S(\boxtimes\Box^k)$ holds.
Then $\triangle\top$ is not refutable in $S$.
\end{prop}

\begin{proof}
Suppose that $\mathsf{E}_S(\boxtimes\Box^k)$ holds.
Since $k\leq k_i$ for every $i<n$, Proposition \ref{prop_hierarchy}(3) yields
$\mathsf{E}_S(\boxtimes\Box^{k_i})$.
Hence, by Proposition \ref{prop_hierarchy}(1),
$\boxtimes\Box^{k_i}\boxtimes\top=\boxtimes\top$.

Since $\Box^{k_n}\top=\top$ by Proposition \ref{propbasicproperty}(1), applications of these equalities give $\triangle\top=\Box^{k_0}\boxtimes\top$.
By Proposition \ref{prop_g2_hierarchy},
$\Box^{k_0}\boxtimes\top$ is not refutable in $S$.
Therefore, $\triangle\top$ is not refutable in $S$.
\end{proof}

Next, we investigate whether the fixed-point assumption in Proposition \ref{prop_ge_general} can be omitted.
We begin with some general consequences of the refutability of values of APS terms.

\begin{lem}\label{lem_refutable}
Let $S$ be an APS, let $\triangle v$ be a one-variable APS term with $d(\triangle) \geq 1$, and let $x\in L$.
If $\triangle x$ is refutable in $S$, then $\boxtimes\top$ is refutable in $S$.
\end{lem}

\begin{proof}
Since $d(\triangle) \geq 1$, it is of the form $\Box^k\boxtimes\triangle_0 v$ for some $k\geq 0$ and some one-variable APS term $\triangle_0 v$.
Suppose that $\triangle x\leq\bot$.
By Proposition \ref{propbasicproperty}(2), we have
$\boxtimes\top\leq \Box^k\boxtimes\triangle_0x=\triangle x\leq\bot$.
Thus $\boxtimes\top$ is refutable in $S$.
\end{proof}

In view of Lemma \ref{lem_refutable}, in order to study the non-refutability of values of one-variable APS terms of positive degree, it suffices to analyze the case where $\boxtimes\top\leq\bot$.
Under this assumption, the behavior of one-variable APS terms becomes much more restricted.
It turns out that the parity of the degree again plays an essential role. 
In fact, the value of a one-variable APS term at $\boxtimes\top$ is completely determined by the parity of its degree.

\begin{lem}\label{lem_parity}
Let $S$ be an APS and suppose that $\boxtimes\top\leq\bot$.
Let $\triangle v$ be a one-variable APS term.
\begin{enumerate}
    \item If $d(\triangle)$ is even, then $\triangle\boxtimes\top=\boxtimes\top$.

    \item If $d(\triangle)$ is odd, then $\triangle\boxtimes\top=\top$.
\end{enumerate}
\end{lem}
\begin{proof}
We first show that $\boxtimes^2\top=\top$ and $\Box\boxtimes\top=\boxtimes\top$.
Since $\boxtimes\top\leq\bot$, (C1) and Proposition \ref{propbasicproperty}(1) give
$\top=\boxtimes\bot\leq\boxtimes^2\top$.
Since $\boxtimes^2\top \leq \top$ by (C5), we get $\boxtimes^2\top=\top$.

By (C4), $\boxtimes\top\leq\Box\boxtimes\top$.
On the other hand, $\Box\boxtimes\top\leq\Box\boxtimes\top$ and $\Box\boxtimes\top\leq\top=\boxtimes^2\top$.
Thus, applying (C3), we obtain $\Box\boxtimes\top\leq\boxtimes\top$.
Therefore $\Box\boxtimes\top=\boxtimes\top$.

We prove the lemma by induction on the construction of $\triangle v$.
For $\triangle v=v$, the statement is immediate.

Suppose that $d(\triangle)$ is even and $\triangle\boxtimes\top=\boxtimes\top$.
Then $d(\Box \triangle)$ is even and $\Box\triangle\boxtimes\top=\Box\boxtimes\top=\boxtimes\top$. 
Also, $d(\boxtimes \triangle)$ is odd and $\boxtimes\triangle\boxtimes\top=\boxtimes^2\top=\top$. 

Suppose that $d(\triangle)$ is odd and $\triangle\boxtimes\top=\top$.
Then $d(\Box \triangle)$ is odd and $\Box\triangle\boxtimes\top=\Box\top=\top$. 
Also, $d(\boxtimes \triangle)$ is even and $\boxtimes\triangle\boxtimes\top=\boxtimes\top$.
\end{proof}

The preceding two lemmas allow us to determine the value of a term of positive degree whenever it is refutable.

\begin{lem}\label{lem_refutable_value}
Let $S$ be an APS, let $\triangle v$ be a one-variable APS term with $d(\triangle) \geq 1$ and let $x\in L$.
If $\triangle x$ is refutable, then $\triangle x=\boxtimes\top$.
\end{lem}

\begin{proof}
Since $d(\triangle)\geq1$, we write $\triangle v=\Box^k\boxtimes\triangle_0v$ for some $k\geq0$ and some one-variable APS term $\triangle_0v$, and let $y=\boxtimes\triangle_0 x$.
Suppose that $\triangle x\leq\bot$.
By Proposition \ref{propbasicproperty}(2),
$y\leq\Box^k y=\triangle x\leq\bot$.
Hence, by (C1) and Proposition \ref{propbasicproperty}(1),
$\top=\boxtimes\bot\leq\boxtimes y$.
Since $y\leq\top$ by (C5), we have $y\leq\boxtimes y$.
Also, by (C4), $y\leq\Box y$.
Therefore, by (C3), $y\leq\boxtimes\top$.

By Proposition \ref{propbasicproperty}(2), $\boxtimes\top\leq y$.
Thus $y=\boxtimes\top$, and hence $\triangle x=\Box^k\boxtimes\top$.

By Lemma \ref{lem_refutable}, $\boxtimes\top\leq\bot$.
Since $\Box^k v$ has even degree, it follows from Lemma \ref{lem_parity}(1) that $\Box^k\boxtimes\top=\boxtimes\top$. 
Therefore $\triangle x=\boxtimes\top$.
\end{proof}

For terms of even degree, a G2-like non-refutability result holds without any fixed-point assumption. 

\begin{thm}\label{thm_even_nonrefutable}
Let $S$ be a consistent APS and let $\triangle v$ be a one-variable APS term of even degree.
Then $\triangle\top$ is not refutable in $S$.
\end{thm}

\begin{proof}
If $d(\triangle)=0$, then $\triangle v=\Box^k v$ for some $k\geq 0$.
Since $\Box^k\top=\top$ by Proposition \ref{propbasicproperty}(1), the refutability of $\triangle\top$ would imply $\top\leq\bot$, contradicting the consistency of $S$.

Suppose that $d(\triangle)\geq 2$ and, toward a contradiction, that $\triangle\top$ is refutable. 
By Lemma \ref{lem_refutable}, $\boxtimes\top\leq\bot$.
Assume that $\triangle v$ is of the form $\triangle_0\boxtimes\Box^k v$, where
$d(\triangle_0)$ is odd.
Since $\Box^k\top=\top$, we have $\triangle\top=\triangle_0\boxtimes\top$.
By Lemma \ref{lem_parity}(2), $\triangle_0\boxtimes\top=\top$.
Hence $\triangle\top=\top$.
Since $\triangle\top$ is refutable, it follows that $\top\leq\bot$, a contradiction. 
\end{proof}

The situation is different for terms of odd degree.
In particular, even the existence of fixed points for all terms of degree three does not yield a G2-like non-refutability result for terms of odd degree evaluated at $\top$.

\begin{prop}\label{prop_no_g2}
There exists a consistent APS $S$ satisfying the following conditions:
\begin{enumerate}
    \item $\mathsf{E}_S(\triangle)$ holds for every one-variable APS term $\triangle v$ of degree three.
    \item For every one-variable APS term $\triangle v$ of odd degree, $\triangle\top$ is refutable in $S$.
\end{enumerate}
\end{prop}
\begin{proof} The desired APS is given in Figure~\ref{fig5}.

\begin{center}
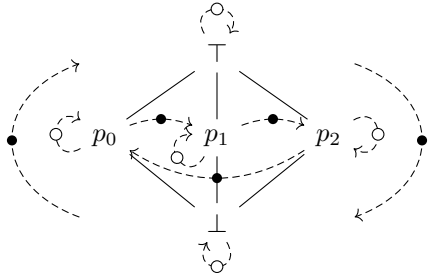

        \begin{tikzcd}
            & \top  \arrow[rd, no head] \arrow[ld, no head] \arrow[box edge, inner loop= 90 ] \arrow[d, no head] \arrow[boxtimes edge, dd,  bend left = 70, shift left = 18, looseness=1.6]& \\ 
            p_0  \arrow[rd, no head] \arrow[box edge, inner loop = 180] \arrow[boxtimes edge, r, bend left = 20 ]& p_1 \arrow[d, no head]\arrow[boxtimes edge, r,  bend left = 20] \arrow[box edge, inner loop = 210]& p_2 \arrow[ld, no head] \arrow[box edge, inner loop = 0] \arrow[boxtimes edge, ll, bend left = 33] \\
            & \bot \arrow[box edge, inner loop = 270] \arrow[boxtimes edge, uu,  bend left = 70, shift left = 18, looseness=1.6]& 
        \end{tikzcd}

\captionof{figure}{APS for Proposition \ref{prop_no_g2}}\label{fig5}
\end{center}

    Since $\Box$ is the identity in this APS and $\boxtimes$ cyclically permutes $p_0, p_1, p_2$, we have $\triangle p_i = \boxtimes^3 p_i = p_i$ for each $i \in \{0, 1, 2\}$. 
    Hence $\mathsf{E}_S (\triangle)$ holds.
    
    Moreover, $\boxtimes$ swaps $\top$ and $\bot$ and $\Box$ fixes them. Hence, since $d(\triangle)$ is odd, we have $\triangle \top = \boxtimes^{d(\triangle)} \top = \boxtimes \top = \bot$.   
\end{proof}

However, a G2-like phenomenon can still be recovered for terms of odd degree if we consider the refutability of their fixed points rather than that of closed terms $\triangle\top$.
Indeed, every fixed point of a term of odd degree is non-refutable in a consistent APS.

\begin{thm}\label{thm_odd_nonrefutable}
Let $S$ be a consistent APS and let $\triangle v$ be a one-variable APS term of odd degree.
If $p$ is a fixed point of $\triangle v$, then $p$ is not refutable in $S$.
\end{thm}
\begin{proof}
Suppose, toward a contradiction, that $p \leq \bot$.
Since $p=\triangle p$, it follows from Lemma \ref{lem_refutable_value} that $p=\boxtimes\top$.
Then $\boxtimes\top\leq\bot$.

Since $d(\triangle)$ is odd, by Lemma \ref{lem_parity}, we have $\triangle\boxtimes\top=\top$.
On the other hand, since $p=\boxtimes\top$ is a fixed point of
$\triangle v$, we have
$\triangle\boxtimes\top=\boxtimes\top$.
Thus $\top=\boxtimes\top\leq\bot$, contradicting the consistency of $S$.
\end{proof}

\section{Meet APSs and L\"ob's Theorem}

So far, we have studied fixed points using only the order structure of an APS and the two unary operations $\Box$ and $\boxtimes$.
In this section, we consider APSs whose underlying partial order is a meet-semilattice.
The meet operation corresponds to conjunction in view of the intended interpretation of APSs.
The existence of meets alone, however, is not sufficient for the analysis developed below.
We therefore consider a strengthening of (C3), denoted by (C3*), that takes the meet operation into account.
After studying fixed points under this additional condition, we turn to a parametrized fixed-point property and an abstract form of L\"ob's theorem.

\begin{defn}
An APS $S=\langle L,\leq,\top,\bot,\Box,\boxtimes\rangle$ is called a \emph{meet APS} if $\langle L,\leq\rangle$ is a meet-semilattice.
For $x,y\in L$, we denote their meet by $x \land y$.
\end{defn}

The presence of the meet operation alone, however, does not provide enough interaction between $\Box$, $\boxtimes$, and $\wedge$ for the fixed-point arguments below.
We therefore consider an additional condition.
If $y$ is provable and $y\land z$ is refutable, then one may expect $z$ itself to be refutable.
This leads to the following condition.

\begin{description}
    \item[(C3*)] If $x \leq \Box y$ and $x \leq \boxtimes(y \land z)$, then $x \leq \boxtimes z$.
\end{description}

Since $L$ is a meet-semilattice, condition (C3*) is equivalently written as
\[
    \Box y\land\boxtimes(y\land z)\leq\boxtimes z.
\]
Thus, under the intended reading $\boxtimes x=\Box\neg x$, condition (C3*) may also be viewed as a counterpart of the axiom $\mathbf{K}: \Box A \land \Box (A \to B) \to \Box B$ of normal modal logic.  

Condition (C3*) is a strengthening of (C3), since (C3) is obtained from (C3*) by taking $z=\top$.
For comparison, condition (C3) can be written as
\[
    \Box y\land\boxtimes y\leq\boxtimes\top.
\]
The converse, however, does not hold in general.

\begin{prop}\label{prop_no_C3ast}
    There exists a meet APS that does not satisfy \textup{(C3*)}. 
\end{prop}
\begin{proof}
Consider the APS given in Figure~\ref{fig3}.
It is based on a Boolean algebra, and hence it is a meet APS.
Since $p \land \neg p = \bot$, we have $\Box p \land \boxtimes(p\land \neg p) = \neg p \land\boxtimes\bot = \neg p$.
On the other hand, $\boxtimes \neg p = \bot$.
Hence $\Box p\land\boxtimes(p\land \neg p)\not\leq\boxtimes \neg p$, so (C3*) fails.

\end{proof}

\subsection{Fixed points of one-variable terms in the original language}

We investigate how the structure of a meet APS together with (C3*) affects the existence of fixed points.
In Section \ref{Sec_existence}, we showed that every one-variable APS term of degree two has a fixed point. 
On the other hand, the existence of fixed points for arbitrary terms of higher even degree remains open in general.
For meet APSs satisfying (C3*), the situation becomes much clearer.

We first consider one-variable APS terms in the original language, that is, terms constructed from $v$ using only $\Box$ and $\boxtimes$.
For meet APSs satisfying (C3*), the parity of the degree yields a complete characterization of the existence of fixed points of such one-variable APS terms.

Before proving the theorem, we prepare the following lemma.

\begin{lem}\label{lem_meet_reduction}
Let $S$ be a meet APS satisfying \textup{(C3*)} and let $x \in L$. 
\begin{enumerate}
    \item For every one-variable APS term $\triangle v$ with $d(\triangle) \geq 1$, we have $\boxtimes \triangle x = \boxtimes \Box \triangle x$. 

    \item For every one-variable APS term $\triangle v$ with $d(\triangle) \geq 1$, we have $\boxtimes \triangle x \leq \boxtimes^3 \triangle x$. 

    \item For every $k\geq 0$, we have $\boxtimes\Box^k\boxtimes x=\boxtimes^2x$.

    \item For every one-variable APS term $\triangle v$ which is of the form $\boxtimes \triangle_0 v$, we have $\triangle \boxtimes x=\boxtimes^{d(\triangle) + 1} x$.

\end{enumerate} 
\end{lem}

\begin{proof}
1. Since $\triangle v$ is of the form $\Box^k \boxtimes \triangle_0 v$ for some $k \geq 0$ and term $\triangle_0 v$, by Proposition \ref{propbasicproperty}(2), $\triangle x \leq \Box \triangle x$. 
By (C1), $\boxtimes \Box \triangle x \leq \boxtimes \triangle x$. 

For the converse, we have $\boxtimes \triangle x \leq \Box \boxtimes \triangle x$ by (C4). 
Also, we have $(\boxtimes \triangle x \land \Box \triangle x) \leq \boxtimes \top$ by (C3). 
By Proposition \ref{propbasicproperty}(2), $\boxtimes \top \leq \triangle x$, and hence $(\boxtimes \triangle x \land \Box \triangle x) \leq \triangle x$. 
By (C1), $\boxtimes \triangle x \leq \boxtimes (\boxtimes \triangle x \land \Box \triangle x)$. 
Applying (C3*) to this and $\boxtimes \triangle x \leq \Box \boxtimes \triangle x$, we obtain $\boxtimes \triangle x \leq \boxtimes \Box \triangle x$.

\medskip

2. 
Since $\boxtimes \triangle x \leq \Box \boxtimes \triangle x$ by (C4), we have $(\boxtimes \triangle x \land \boxtimes \boxtimes \triangle x) \leq (\Box \boxtimes \triangle x \land \boxtimes \boxtimes \triangle x) \leq \boxtimes \top$ by (C3). 
By Proposition \ref{propbasicproperty}(2), $\boxtimes \top \leq \triangle x$ because $d(\triangle) \geq 1$. 
Hence $(\boxtimes \triangle x \land \boxtimes \boxtimes \triangle x) \leq \triangle x$. 
By (C1), $\boxtimes \triangle x \leq \boxtimes (\boxtimes \triangle x \land \boxtimes \boxtimes \triangle x)$. 
Applying (C3*) to this and $\boxtimes \triangle x \leq \Box \boxtimes \triangle x$, we obtain $\boxtimes \triangle x \leq \boxtimes \boxtimes \boxtimes \triangle x$.

\medskip

3. From (1), for every $j\geq0$, $\boxtimes\Box^j \boxtimes x = \boxtimes\Box^{j+1} \boxtimes x$ by taking $\triangle v = \Box^j \boxtimes v$.
It follows that $\boxtimes \boxtimes x = \boxtimes \Box \boxtimes x = \boxtimes \Box^2 \boxtimes x = \cdots = \boxtimes\Box^k\boxtimes x$ for every $k\geq0$.

\medskip

4. We can write
\[
    \triangle v= \boxtimes\Box^{k_1}\boxtimes\Box^{k_2} \cdots \boxtimes \Box^{k_n}v, 
\]
    where $n = d(\triangle)$.
    By applying (3) repeatedly, we obtain $\triangle \boxtimes x = \boxtimes^{n+1} x$. 
\end{proof}

\begin{thm}\label{thm_meet_parity}
Let $S$ be a meet APS satisfying \textup{(C3*)}, and let $\triangle v$ be a
one-variable APS term in the original language.
Then the following statements hold:
\begin{enumerate}
    \item If $d(\triangle)$ is even, then $\mathsf{E}_S(\triangle)$ holds.

    \item If $d(\triangle)$ is odd, then
    $\mathsf{E}_S(\triangle) \iff \mathsf{E}_S(\boxtimes)$.
\end{enumerate}
\end{thm}
\begin{proof}
1. If $d(\triangle)=0$, then $\triangle v=\Box^k v$ for some $k\geq0$,
and $\top$ is a fixed point of $\triangle v$ by Proposition \ref{propbasicproperty}(1).

Suppose that $d(\triangle)=n\geq2$.
By Proposition \ref{circulation}, we may assume that $\triangle v$ is of the form $\boxtimes \triangle_0 v$ for some term $\triangle_0 v$. 
By Lemma \ref{lem_meet_reduction}(4), $\triangle \boxtimes\top = \boxtimes^{d(\triangle)+1}\top$.
By Theorem \ref{2boxtimesfp}(1), we have $\boxtimes^3\top=\boxtimes\top$. 
Hence $\boxtimes^{d(\triangle)+1}\top=\boxtimes\top$ since $d(\triangle)+1$ is odd.
Therefore $\triangle \boxtimes\top = \boxtimes\top$, and so $\mathsf{E}_S(\triangle)$ holds.

\medskip

2. $(\Leftarrow)$: The implication $\mathsf{E}_S(\boxtimes)\Rightarrow\mathsf{E}_S(\triangle)$ follows from Corollary \ref{cor_hierarchy_godel}.

\medskip

$(\Rightarrow)$: 
Suppose that $d(\triangle)$ is odd and $\mathsf{E}_S(\triangle)$ holds.
By Proposition \ref{circulation}, we may assume that $\triangle v$ is of the form $\boxtimes \triangle_0 v$ for some term $\triangle_0 v$. 
Let $p \in L$ be a fixed point of $\triangle v$, that is, $p = \triangle p$.
By Lemma \ref{lem_meet_reduction}(4), we get $p = \boxtimes^{d(\triangle)} p$ because
\[
     p = \triangle p = \triangle \boxtimes \triangle_0 p = \boxtimes^{d(\triangle)+1} \triangle_0 p = \boxtimes^{d(\triangle)} \boxtimes \triangle_0 p = \boxtimes^{d(\triangle)} \triangle p = \boxtimes^{d(\triangle)} p.  
\]
For $i \geq 2$, we have $\boxtimes^i p \leq \boxtimes^{i+2} p$ because it follows from Lemma \ref{lem_meet_reduction}(2) that
\[
    \boxtimes^i p = \boxtimes (\boxtimes^{i-1} p) \leq \boxtimes^3 (\boxtimes^{i-1} p) = \boxtimes^{i+2} p. 
\]
The cases $i = 0, 1$ follow similarly from Lemma \ref{lem_meet_reduction}(2), using $p = \boxtimes \triangle_0 p = \boxtimes \triangle_0 \boxtimes \triangle_0 p$. 

Since $d(\triangle)$ is odd, we have
\[
    p \leq \boxtimes^2 p \leq\cdots\leq \boxtimes^{d(\triangle)+1} p = \boxtimes p \leq \boxtimes^3 p\leq\cdots\leq \boxtimes^{d(\triangle)} p = p.
\]
Thus $p = \boxtimes p$.
Therefore $\mathsf{E}_S(\boxtimes)$ holds.
\end{proof}

The existence result in Theorem \ref{thm_meet_parity}(1) cannot be strengthened to uniqueness in general.

\begin{prop}
There exist a meet APS $S$ satisfying \textup{(C3*)} and a one-variable APS term $\triangle v$ of even degree such that $\mathsf{U}_S(\triangle)$ does not hold.
\end{prop}

\begin{proof}
Consider the APS in Figure \ref{fig2}.
It is a meet APS satisfying (C3*), and both $\top$ and $\bot$ are fixed points of $\boxtimes^2v$.
Thus $\mathsf{U}_S(\boxtimes^2)$ does not hold.
\end{proof}

\subsection{Fixed points of one-variable meet APS terms}

A \emph{one-variable meet APS term} is generated from the variable $v$ by finitely many applications of $\Box$, $\boxtimes$, and $\land$.
We extend the notion of degree to such terms as follows. 

\begin{defn}[degree]
The \emph{degree} $d(\triangle)$ of a one-variable meet APS term $\triangle v$ is recursively defined as follows: 
\begin{enumerate}
    \item $d(v)=0$, 
    \item $d(\Box\triangle_0)=d(\triangle_0)$, 
    \item $d(\boxtimes\triangle_0)=d(\triangle_0)+1$, 
    \item $d(\triangle_0\land\triangle_1) = \max\{d(\triangle_0),d(\triangle_1)\}$. 
\end{enumerate}
\end{defn}

For the analysis of one-variable meet APS terms, the following parameter will play an important role.

\begin{defn}
For each one-variable meet APS term $\triangle v$, we define $\mu(\triangle)\in\mathbb{N}\cup\{\infty\}$ recursively as follows:
\begin{enumerate}
    \item $\mu(v)=\infty$, 
    \item $\mu(\Box\triangle)=\mu(\triangle)+1$, 
    \item $\mu(\boxtimes\triangle)=0$, 
    \item $\mu(\triangle_0\land\triangle_1) = \min\{\mu(\triangle_0), \mu(\triangle_1)\}$.
\end{enumerate}
Here, $\infty+1=\infty$.
\end{defn}

In particular, $\mu(\triangle)=\infty$ if and only if $d(\triangle)=0$.
We first use this parameter to extend the existence part of Corollary \ref{cor_hierarchy_godel} to one-variable meet APS terms.

\begin{prop}\label{prop_meet_godel}
Let $S$ be a meet APS. 
Then for every one-variable meet APS term $\triangle v$, 
\[
    \mathsf{E}_S(\boxtimes) \Rightarrow \mathsf{E}_S(\triangle). 
\]
\end{prop}

\begin{proof}
Suppose that $\mathsf{E}_S(\boxtimes)$ holds. 
By Proposition \ref{prop_hierarchy}(3), $\mathsf{E}_S(\boxtimes \Box^k)$ holds for every $k \geq 0$. 
By Proposition \ref{prop_hierarchy}(1), we obtain $\boxtimes \Box^k \boxtimes \top = \boxtimes \top$ for every $k \geq 0$. 

We prove by induction on the construction of $\triangle v$ that there exists a natural number $a$ such that 
\[
    \triangle \Box^k \boxtimes \top = \Box^{\min\{k+a, \mu(\triangle)\}} \boxtimes \top
\]
for every $k \geq 0$, where $\min\{k+a,\infty\}=k+a$.

For $\triangle v=v$, we have $\mu(\triangle) = \infty$ and $\triangle \Box^k \boxtimes \top = \Box^{\min\{k+0, \infty\}} \boxtimes \top$. 
So let $a=0$.

Suppose that the statement holds for $\triangle_0 v$ and $\triangle_1 v$ with parameters $a_0$ and $a_1$, respectively. 

If $\triangle v$ is $\Box\triangle_0 v$, then $\mu(\triangle) = \mu(\triangle_0) + 1$ and 
\[
    \Box \triangle_0 \Box^k \boxtimes \top = \Box \Box^{\min\{k+a_0, \mu(\triangle_0)\}} \boxtimes \top = \Box^{\min\{k+a_0+1, \mu(\triangle_0)+1\}} \boxtimes \top.
\]
Hence let $a = a_0+1$. 

If $\triangle v$ is $\boxtimes\triangle_0 v$, then $\mu(\triangle) = 0$ and 
\[
    \boxtimes \triangle_0 \Box^k \boxtimes \top = \boxtimes \Box^{\min\{k+a_0, \mu(\triangle_0)\}} \boxtimes \top = \boxtimes \top.
\]
Thus, let $a = 0$. 

If $\triangle v$ is $\triangle_0 v\land\triangle_1 v$, then let $a = \min\{a_0,a_1\}$ because we have $\mu(\triangle) = \min \{\mu(\triangle_0), \mu(\triangle_1)\}$ and 
\begin{align*}
    \triangle_0 \Box^k \boxtimes \top \land \triangle_1 \Box^k \boxtimes \top & = \Box^{\min\{k+a_0, \mu(\triangle_0)\}} \boxtimes \top \land \Box^{\min\{k+a_1, \mu(\triangle_1)\}} \boxtimes \top \\
    & = \Box^{\min\{k+a, \mu(\triangle)\}} \boxtimes \top. 
\end{align*}
Here, we used the inequality $\Box^i \boxtimes \top \leq \Box^j \boxtimes \top$ for $i \leq j$ that is obtained from Proposition \ref{propbasicproperty}(2). 

Let $\triangle v$ be a one-variable meet APS term. 
We find a parameter $a$ for $\triangle v$ fulfilling the above condition. 

Suppose first that $\mu(\triangle) = \infty$.
Then $\triangle v$ contains no occurrence of $\boxtimes$.
Since $\Box\top=\top$, it is easily shown that $\triangle\top=\top$.
Thus $\top$ is a fixed point of $\triangle v$.

Suppose next that $\mu(\triangle) < \infty$.
Then
\[
    \triangle \Box^{\mu(\triangle)} \boxtimes \top = \Box^{\min\{\mu(\triangle) + a, \mu(\triangle)\}} \boxtimes \top = \Box^{\mu(\triangle)} \boxtimes \top.
\]
Hence $\Box^{\mu(\triangle)} \boxtimes \top$ is a fixed point of $\triangle v$.

In either case, we obtain that $\mathsf{E}_S(\triangle)$ holds. 
\end{proof}

Proposition \ref{prop_meet_godel} extends the existence part of Corollary \ref{cor_hierarchy_godel} to one-variable meet APS terms.
We emphasize that (C3*) is not assumed in this proposition.
The uniqueness part, however, does not extend in general, even under (C3*).

\begin{prop}\label{prop_meet_nonunique}
There exist a meet APS $S$ satisfying \textup{(C3*)} and a one-variable meet APS term $\triangle v$ such that $\mathsf{E}_S(\boxtimes)$ holds but $\mathsf{U}_S(\triangle)$ does not hold.
\end{prop}
\begin{proof}

    The desired APS is given in Figure~\ref{fig7}.

\begin{center}
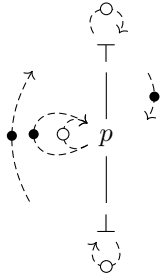

        \begin{tikzcd}
             \top \arrow [d, no head]
             \arrow[box edge, inner loop = 90]
             \arrow[boxtimes edge, d , bend left=30, shift left = 5]\\ 
             p \arrow[d, no head] \arrow[box edge, inner loop = 180] \arrow[boxtimes edge, outer loop = 180]\\
             \bot \arrow[box edge, inner loop = 270] \arrow[boxtimes edge, uu, bend left = 30, shift left = 10 ]
        \end{tikzcd}

\captionof{figure}{APS for Proposition \ref{prop_meet_nonunique}}\label{fig7}
\end{center}

 We have $\boxtimes p = p$ and therefore $\mathsf{E}_S(\boxtimes)$ holds. On the other hand, let $\triangle v := v \land \boxtimes v$. Then both $p$ and $\bot$ are distinct fixed points of $\triangle v$; $\mathsf{U}_S(\triangle)$ does not hold.
\end{proof}

We next ask whether Theorem \ref{thm_meet_parity} extends to one-variable meet APS terms.
The answer is negative, that is, even a one-variable meet APS term of degree two need not have a fixed point.

\begin{prop}\label{prop_meet_degree_two_no_fp}
There exist a meet APS $S$ satisfying \textup{(C3*)} and a one-variable meet APS term $\triangle v$ of degree two such that $\mathsf{E}_S(\triangle)$ does not hold.
\end{prop}

\begin{proof}
Consider the APS in Figure~\ref{fig2}, which satisfies (C3*). 
Let $\triangle v := \boxtimes v \land \boxtimes (v \land \boxtimes v)$, which is of degree two. It is easy to check that  $\triangle \top = \bot$ and $\triangle \bot = \top$. Hence $\triangle v$ has no fixed point, that is, $\mathsf{E}_S(\triangle)$ does not hold.
\end{proof}

Although Theorem \ref{thm_meet_parity}(1) does not extend to arbitrary one-variable meet APS terms, it can be extended under an appropriate parity condition.
We first show that, under this condition, the value of a one-variable meet APS term at $\boxtimes\top$ is determined by its degree and the parameter $\mu$.
For this purpose, we prepare the following lemma.

\begin{lem}\label{lem_meet_reduction2}
Let $S$ be a meet APS satisfying \textup{(C3*)}, and let $\triangle v$ be a one-variable meet APS term.
Suppose that, for every subterm of the form $\triangle_0 v\land\triangle_1 v$, the degrees $d(\triangle_0)$ and $d(\triangle_1)$ have the same parity.
Then there exists a natural number $a$ such that, for every $k\geq0$,
the following statements hold:
\begin{enumerate}
    \item If $d(\triangle)$ is even, then $\triangle\Box^k\boxtimes\top = \Box^{\min\{k+a,\mu(\triangle)\}}\boxtimes\top$. 

    \item If $d(\triangle)$ is odd, then $\triangle\Box^k\boxtimes\top = \Box^{\min\{k+a,\mu(\triangle)\}}\boxtimes^2\top$. 
\end{enumerate}
\end{lem}

\begin{proof}
We prove the two statements simultaneously by induction on the construction
of $\triangle v$.

If $\triangle v=v$, then $d(\triangle)=0$ and $\mu(\triangle) = \infty$.
Since $\triangle \Box^k \boxtimes \top = \Box^{\min \{k + 0, \infty\}} \boxtimes\top$, let $a = 0$. 

Suppose that the statement holds for $\triangle_0 v$ and $\triangle_1 v$ with parameters $a_0$ and $a_1$, respectively.

If $\triangle v$ is $\Box\triangle_0 v$, then $\mu(\triangle) = \mu(\triangle_0) + 1$. 
If $d(\triangle)$ is even, then $d(\triangle_0)$ is also even. 
By the induction hypothesis,  
\[
    \Box \triangle_0 \Box^k \boxtimes \top = \Box \Box^{\min\{k+a_0, \mu(\triangle_0)\}} \boxtimes \top = \Box^{\min\{k+a_0+1, \mu(\triangle_0)+1\}} \boxtimes \top.
\]
So, let $a = a_0+1$. 
The odd case is proved in the same way by letting $a = a_0 + 1$, with $\boxtimes^2\top$ in place of $\boxtimes\top$.

If $\triangle v=\boxtimes\triangle_0v$, then $\mu(\triangle)=0$. 
If $d(\triangle)$ is odd, then $d(\triangle_0)$ is even. 
By the induction hypothesis and Lemma \ref{lem_meet_reduction}(3),
\[
    \triangle\Box^k\boxtimes\top = \boxtimes \Box^{\min\{k+a_0, \mu(\triangle_0)\}}\boxtimes\top = \boxtimes^2\top = \Box^{\min\{k+0, 0\}} \boxtimes^2 \top. 
\]
Hence let $a = 0$. 

If $d(\triangle)$ is even, then $d(\triangle_0)$ is odd. 
Similarly, we have $\triangle\Box^k\boxtimes\top = \boxtimes\top$.
Thus let $a = 0$. 

If $\triangle v=\triangle_0v\land\triangle_1v$, then $\mu(\triangle) = \min \{\mu(\triangle_0), \mu(\triangle_1)\}$. 
By the assumption, $d(\triangle_0)$ and $d(\triangle_1)$ have the same
parity.

If $d(\triangle)$ is even, then both $d(\triangle_0)$ and $d(\triangle_1)$ are even. 
By the induction hypothesis, 
\begin{align*}
    \triangle\Box^k\boxtimes\top & = \Box^{\min\{k+a_0,\mu(\triangle_0)\}}\boxtimes\top
    \land \Box^{\min\{k+a_1,\mu(\triangle_1)\}}\boxtimes\top\\
    &= \Box^{\min\{k+a,\mu(\triangle)\}}\boxtimes\top.
\end{align*}
So, let $a=\min\{a_0,a_1\}$.
The odd case is proved in the same way by letting $a=\min\{a_0,a_1\}$.
\end{proof}

Lemma \ref{lem_meet_reduction2} allows us to extend Theorem \ref{thm_meet_parity}(1) to a class of one-variable meet APS terms.

\begin{thm}\label{thm_meet_even}
Let $S$ be a meet APS satisfying \textup{(C3*)}, and let $\triangle v$
be a one-variable meet APS term.
Suppose that, for every subterm of the form
$\triangle_0v\land\triangle_1v$, the degrees
$d(\triangle_0)$ and $d(\triangle_1)$ have the same parity.
If $d(\triangle)$ is even, then $\mathsf{E}_S(\triangle)$ holds.
\end{thm}

\begin{proof}
If $d(\triangle)=0$, then $\triangle v$ contains no occurrence of
$\boxtimes$, and hence $\triangle\top=\top$.

Suppose that $d(\triangle)>0$.
Then $\mu(\triangle) < \infty$.
By Lemma \ref{lem_meet_reduction2}, there exists a natural number $a$ such that
\[
    \triangle\Box^{\mu(\triangle)}\boxtimes\top = \Box^{\min\{\mu(\triangle)+a,\mu(\triangle)\}}\boxtimes\top = \Box^{\mu(\triangle)}\boxtimes\top.
\]
Thus $\Box^{\mu(\triangle)} \boxtimes\top$ is a fixed point of $\triangle v$.
\end{proof}

It remains open whether the other results obtained for one-variable APS terms in the original language extend to suitable one-variable meet APS terms.

\begin{prob}\label{prob_meet_terms}
Let $S$ be a meet APS satisfying \textup{(C3*)}, and let $\triangle v$ be a one-variable meet APS term such that, for every subterm of the form $\triangle_0 v \land \triangle_1 v$, the degrees $d(\triangle_0)$ and $d(\triangle_1)$ have the same parity.
\begin{enumerate}
    \item If $d(\triangle)$ is odd and $\mathsf{E}_S(\triangle)$ holds, does $\mathsf{E}_S(\boxtimes)$ hold?

    \item If $S$ is consistent and $d(\triangle)$ is even, then is $\triangle\top$ non-refutable in $S$?

    \item If $S$ is consistent, $d(\triangle)$ is odd, and $p$ is a fixed point of $\triangle v$, then is $p$ non-refutable in $S$?
\end{enumerate}
\end{prob}

\subsection{L\"ob's theorem}

A \emph{two-variable meet APS term} $\triangle(v,w)$ is a term constructed
from the variables $v$ and $w$ using $\Box$, $\boxtimes$, and $\land$.
We show that the existence of fixed points for a particular two-variable meet APS term is sufficient to derive an abstract form of L\"ob's theorem.
More precisely, we consider the term $\boxtimes(v\land w)$ and the following fixed-point condition:
\begin{description}
    \item[(FP)] For every $x\in L$, there exists $p\in L$ such that $p=\boxtimes(p\land x)$.
\end{description}

We also consider the following condition:
\begin{description}
    \item[(L\"ob)] For every $x\in L$, if $(\boxtimes x\land x)\leq\bot$, then $x \leq \bot$.
\end{description}

This condition is an abstract counterpart of L\"ob's theorem.
Indeed, under the intended interpretation $\boxtimes x=\Box\neg x$, (L\"ob) corresponds to L\"ob's theorem 
\[
    \text{if }\vdash\Box A\to A,\text{ then }\vdash A.
\]

\begin{thm}\label{thm_lob}
Let $S$ be a meet APS satisfying \textup{(C3*)} and \textup{(FP)}.
Then $S$ satisfies \textup{(L\"ob)}.
\end{thm}

\begin{proof}
Suppose that $(\boxtimes x\land x)\leq\bot$. 
By (FP), there exists $p\in L$ such that $p = \boxtimes(p\land x)$.
By (C4), $p\leq\Box p$.
Hence, by (C3*), $p\leq\boxtimes x$.
It follows that $(p\land x) \leq (\boxtimes x\land x) \leq \bot$.
By (C1), $\boxtimes\bot \leq \boxtimes(p\land x) = p$.
Since $\boxtimes\bot=\top$, we obtain $\top\leq p$.
Therefore $x = (\top\land x) \leq (p\land x) \leq \bot$.
Thus $S$ satisfies (L\"ob).
\end{proof}

The condition (FP) is stronger than the existence of a G\"odelian fixed point alone.

\begin{prop}\label{prop_FP_godel}
Let $S$ be a meet APS satisfying \textup{(FP)}.
Then $\mathsf{E}_S(\boxtimes)$ holds.
\end{prop}

\begin{proof}
Applying \textup{(FP)} to $\top$, we find $p\in L$ such that $p=\boxtimes(p\land\top)$.
Since $(p\land\top)=p$, we obtain $p=\boxtimes p$.
Thus $\mathsf{E}_S(\boxtimes)$ holds.
\end{proof}

Conversely, the existence of a G\"odelian fixed point yields a weaker
version of \textup{(FP)}.

\begin{prop}\label{prop_godel_weak_FP}
Let $S$ be a meet APS.
If $\mathsf{E}_S(\boxtimes)$ holds, then, for every $x\in L$, there exists
$p\in L$ such that $p=\boxtimes(p\land\boxtimes x)$.
\end{prop}

\begin{proof}
Suppose that $\mathsf{E}_S(\boxtimes)$ holds. 
By Theorem \ref{BS2}(2), $\boxtimes^2\top=\boxtimes\top$.
Since $\boxtimes\top\leq\boxtimes x$, we have $\boxtimes\top\land\boxtimes x=\boxtimes\top$.
Therefore $\boxtimes(\boxtimes\top\land\boxtimes x) = \boxtimes^2\top = \boxtimes\top$.
\end{proof}

Conversely, (L\"ob) together with (C3*) does not even guarantee the existence of fixed points for terms of odd degree.

\begin{prop}\label{prop_lob_nofp}
There exists a consistent meet APS $S$ satisfying \textup{(C3*)} and \textup{(L\"ob)} such that $\mathsf{E}_S(\triangle)$ fails for every one-variable APS term $\triangle v$ of odd degree.
\end{prop}

\begin{proof}
The desired APS is given in Figure~\ref{fig8}.

\begin{center}
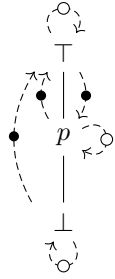

    \begin{tikzcd}
             \top \arrow [d, no head] \arrow[box edge, inner loop = 90] \arrow[boxtimes edge, d , bend left=30, shift left = 1 ]\\ 
             p \arrow[d, no head]  \arrow[boxtimes edge, u, bend left, shift left = 1] \arrow[box edge, inner loop = 0, yshift=-0.5ex] \\
             \bot \arrow[box edge, inner loop = 270] \arrow[boxtimes edge, uu, bend left = 30, shift left = 3] 
        \end{tikzcd}
\captionof{figure}{APS for Proposition \ref{prop_lob_nofp}}\label{fig8}
\end{center}

It is easy to check that this is a consistent meet APS satisfying (C3*) and (L\"ob). 
Let $\triangle v$ be a one-variable APS term of odd degree. 
Since $\Box$ is the identity, we have $\triangle x = \boxtimes^{d(\triangle)} x$ for every $x \in L$. 
Moreover, $\boxtimes x \in \{\top, p\}$ for every $x \in L$, and $\boxtimes$ swaps these two elements. 
Hence $\boxtimes^{3}v = \boxtimes v$. Combining these with the assumption that $d(\triangle)$ is odd, we obtain $\triangle v = \boxtimes v$. 
Since $\boxtimes$ has no fixed point, neither does $\triangle$.
\end{proof}

This contrasts with the situation in the G\"odel--L\"ob provability logic $\mathbf{GL}$.
By the de Jongh--Sambin Fixed Point theorem, every formula modalized in a variable has a fixed point in $\mathbf{GL}$.
Proposition \ref{prop_lob_nofp} shows that meet APSs satisfying L\"ob's theorem are not sufficient to recover the corresponding fixed-point phenomenon. 

\section{Concluding remarks}

In this paper, we have studied fixed points of APS terms and their connections with G2-like phenomena. 
Our main results can be summarized as follows. 
\begin{itemize}
    \item Every one-variable APS term of degree two has a fixed point, and explicit fixed points can be given for all such terms (Theorem \ref{2boxtimesfp}).
   
    \item The existence properties $\mathsf{E}_S(\boxtimes\Box^k)$ form a strict hierarchy (Propositions \ref{prop_hierarchy} and \ref{prop_strict}).
    Moreover, an appropriate level of this hierarchy guarantees the existence and uniqueness of fixed points for arbitrary one-variable APS terms of positive degree (Theorem \ref{thm_collapse} and Corollary \ref{cor_hierarchy_godel}).

    \item The hierarchy of the existence of fixed points yields corresponding G2-like non-refutability results (Propositions \ref{prop_g2_hierarchy} and \ref{prop_ge_general}).
    More generally, every term of even degree satisfies a non-refutability result without any fixed-point assumption (Theorem \ref{thm_even_nonrefutable}).
    On the other hand, Proposition \ref{prop_no_g2} shows that the analogous non-refutability result for terms of odd degree fails in general. 
    However, every fixed point of a term of odd degree is non-refutable (Theorem \ref{thm_odd_nonrefutable}).
    
    \item In meet APSs satisfying (C3*), every one-variable APS term in the original language of even degree has a fixed point, while a term of odd degree has a fixed point if and only if $\mathsf{E}_S(\boxtimes)$ holds (Theorem \ref{thm_meet_parity}). 
    Also, $\mathsf{E}_S(\boxtimes)$ implies the existence of a fixed point for every one-variable meet APS term (Proposition \ref{prop_meet_godel}). 
    Moreover, under the parity condition on meet subterms, every one-variable meet APS term of even degree has a fixed point in meet APSs satisfying \textup{(C3*)} (Theorem \ref{thm_meet_even}).

    \item Finally, the parametrized fixed-point condition \textup{(FP)} implies the abstract L\"ob condition \textup{(L\"ob)} in meet APSs satisfying (C3*) (Theorem \ref{thm_lob}).
    Although (FP) implies the existence of a G\"odelian fixed point (Proposition \ref{prop_FP_godel}), L\"ob's theorem does not conversely guarantee fixed points even for terms of odd degree (Proposition \ref{prop_lob_nofp}).
\end{itemize}

These results show that APSs provide a weak order-theoretic framework in which many fixed-point and G2-like phenomena can still be studied.
At the same time, some implications that hold in arithmetic and provability logic fail for APSs.
Thus APSs help us see which phenomena follow from the basic interaction between provability and refutability and which require additional assumptions.
To clarify further the position of APSs and to address problems that could not be analyzed within the present paper, the first author is developing proof-theoretic and relational-semantic analyses of APSs.
These results will be presented in a forthcoming paper.

\section*{Acknowledgments}

Some of the examples of APSs constructed in this paper were found through discussions with Claude (Anthropic) and ChatGPT (OpenAI). 
The authors verified all such examples and their proofs independently.
The second author was supported by JSPS KAKENHI Grant Number JP23K03200.

\bibliographystyle{plain}
\bibliography{references}

\begin{thebibliography}{1}

\bibitem{AB05}
Sergei~N. Artemov and Lev~D. Beklemishev.
\newblock Provability logic.
\newblock In D.~Gabbay and F.~Guenthner, editors, {\em Handbook of
  Philosophical Logic}, volume~13, pages 189--360. Springer, Dordrecht, 2nd
  edition, 2005.

\bibitem{beklemishev2016some}
L~Beklemishev and D~Shamkanov.
\newblock Some abstract versions of {G\"o}del’s second incompleteness theorem
  based on non-classical logics.
\newblock In {\em Liber Amicorum Alberti: A tribute to Albert Visser},
  volume~30, pages 15--29, 2016.

\bibitem{Bern76}
Claudio Bernardi.
\newblock The uniqueness of the fixed-point in every diagonalizable algebra.
\newblock {\em Studia Logica}, 35(4):335--343, 1976.

\bibitem{Bool93}
George {Boolos}.
\newblock {\em {The logic of provability}}.
\newblock Cambridge: Cambridge University Press, 1993.

\bibitem{DP02}
B.~A. Davey and H.~A. Priestley.
\newblock {\em Introduction to lattices and order.}
\newblock Cambridge: Cambridge University Press, 2nd edition, 2002.

\bibitem{Lind03}
Per Lindstr{\"o}m.
\newblock {\em Aspects of Incompleteness}.
\newblock Number 10 in Lecture Notes in Logic. A K Peters, 2nd edition, 2003.

\bibitem{Samb76}
Giovanni Sambin.
\newblock An effective fixed-point theorem in intuitionistic diagonalizable
  algebras.
\newblock {\em Studia Logica}, 35(4):345--361, 1976.

\end{thebibliography}

\end{document}